\documentclass[notitlepage]{amsart}

\usepackage[all]{xy}
\usepackage{pb-diagram,pb-xy}

\usepackage{amsmath,amscd}
\usepackage{enumerate}
\usepackage{amsmath}
\usepackage{amssymb}
\usepackage{amsfonts}
\usepackage{eufrak}
\usepackage{latexsym}
\usepackage{leftidx}
\usepackage{amsthm}
\usepackage{color}

\newtheorem{theorem}{Theorem}[section]
\newtheorem{proposition}[theorem]{Proposition}
\newtheorem{lemma}[theorem]{Lemma}
\newtheorem{defin}[theorem]{Definition}
\newtheorem{corollary}[theorem]{Corollary}
\newtheorem{remark}[theorem]{Remark}
\newtheorem{example}[theorem]{Example}

\newcommand{\mini}{{\mathrm{min}}}

\newcommand{\modu}{{\mathrm{mod}}}

\newcommand{\Ker}{{\mathrm{Ker}}}
\newcommand{\Coker}{{\mathrm{Coker}}}
\newcommand{\Hom}{{\mathrm{Hom}}}

\newcommand{\Tr}{{\mathrm{Tr}}}
\newcommand{\End}{{\mathrm{End}}}
\newcommand{\Ext}{{\mathrm{Ext}}}
\newcommand{\Tor}{{\mathrm{Tor}}}
\newcommand{\proj}{{\mathrm{proj}}}

\newcommand{\F}{{\mathcal{F}}}

\newcommand{\Q}{{\mathbf{Q}}}

\newcommand{\C}{{\mathcal{C}}}

\newcommand{\add}{{\mathrm{add}}}
\newcommand{\rad}{{\mathrm{rad}}}
\newcommand{\Top}{{\mathrm{top}}}

\newcommand{\ind}{{\mathrm{ind}}}
\renewenvironment{proof}{\noindent \bf Proof. \rm}{$\quad \hfill \square$ \bigskip}
\begin{document}


\title[existence and construction of proper costratifying systems]
{On the existence and construction of proper costratifying systems}

\author{\sc O. Mendoza, M. I. Platzeck and M. Verdecchia}


\maketitle

\begin{abstract}
In this paper we further study the notion of proper costratifying systems, defined in \cite{MPV}. We give sufficient conditions for their existence, and investigate the relation between the stratifying systems defined by K. Erdmann and C. S\'aenz in \cite{ES} with the proper costratifying systems.
\end{abstract}

\section*{Introduction.}

In \cite{MPV} we define and study the notion of a proper costratifying
system, which is a generalization of the so called proper costandard modules to the
context of stratifying systems.
  \

K. Erdmann and C. S\'aenz defined in \cite{ES} the notion of a stratifying system
$(\Theta=\{\Theta(i)\}_{i=1}^{t},\underline{Y},\leq),$ and proved that such a system satisfies that each $\Theta(i)$ is indecomposable
and $\Ext_\Lambda ^{1}(\Theta(i),\Theta(j))=0$ for $i\geq j.$ Reciprocally, they also showed that given a family of indecomposable $\Lambda$-modules $\Theta=\{\Theta(i)\}_{i=1}^{t}$ satisfying that $\Ext_\Lambda ^{1}(\Theta(i),\Theta(j))=0$ for $i\geq j,$ and such that $\Hom_\Lambda(\Theta(i),\Theta(j))=0$ for $i>j,$ there is a stratifying system $(\Theta,\underline{Y},\leq)$.
\

For proper costratifying systems, the situation is different. On the one hand, it is true that the $\Psi(i)$'s of a proper costratifying system
$(\Psi,\Q,\leq)$ satisfy that $\Ext^{1}_\Lambda(\Psi(i),\Psi(j))=0$ for $i<j$ (see \cite[Lemma 3.8]{MPV}). However, the existence of a family
$\Psi=\{\Psi(i)\}_{i=1}^{t}$ such that $\Hom_\Lambda(\Psi(i),\Psi(j))=0=\Ext^{1}_\Lambda(\Psi(i),\Psi(j))$ for $i<j$, does not ensure the
existence of a proper costratifying system $(\Psi,\Q,\leq),$ even if we assume that $\End_\Lambda(\Psi(i))$ is a division ring for all $i\in[1,t]$.
\

In this paper we prove an existence result for proper costratifying systems in this direction. To this end, we assume as an additional hypothesis that the length of the indecomposable modules filtered by $\Psi$ is uniformly bounded.
\

Let $P(1), ..., P(n)$ be an ordered sequence of the non-isomorphic indecomposable projective modules over an artin algebra $\Lambda$. By definition, the standard module ${}_\Lambda\Delta(i)$ is
the largest factor module of $P(i)$ with composition factors only amongst $S(1), ..., S(i)$, where $S(j)$ is the simple top of $P(j)$.
Denote by mod$(\Lambda)$ the category of finitely generated left $\Lambda$-modules. Let $\mathcal{F}({}_\Lambda\Delta)$ denote the subcategory
of mod$(\Lambda)$ consisting of the $\Lambda$-modules having a filtration with factors isomorphic to standard modules. The algebra
$\Lambda$ is said to be standardly stratified if all projective $\Lambda$-modules belong to $\mathcal{F}({}_\Lambda\Delta)$
(see \cite{CPS}, \cite{ADL}, \cite{AHLU}, \cite{Webb}, \cite{ES}, \cite{Platzeck Reiten}, \cite{Xi}).
\

For an artin algebra $\Lambda$ there exists another family of modules which plays an important role: the proper standard (respectively, the proper costandard) modules ${}_\Lambda\overline{\Delta}(i)$ (respectively, ${}_\Lambda\overline{\nabla}(i)$), defined as some appropriated factors of the ${}_\Lambda\Delta(i)$ (respectively, submodules of the ${}_\Lambda\nabla(i)$). These modules were defined by V. Dlab and have the property that $\Lambda$ is a standardly stratified algebra, that is, all projective $\Lambda$-modules belong to $\mathcal{F}({}_\Lambda\Delta)$, if and only if all injective $\Lambda$-modules belong to $\mathcal{F}({}_\Lambda\overline{\nabla})$ (see \cite{Dlab} and \cite{L}). Here $\mathcal{F}({}_\Lambda\overline{\nabla})$ denotes the subcategory
of mod$(\Lambda)$ consisting of the $\Lambda$-modules having a filtration with factors isomorphic to proper costandard modules.
\

In connection with the study of standardly stratified algebras, I. Agoston, D. Happel, E. Luk\'acs and L. Unger showed, for such an algebra $(\Lambda,\leq)$, that $\mathcal{F}({}_\Lambda\overline{\nabla})=\mathcal{F}({}_\Lambda\Delta)^{\bot_1}=\{M\in\modu(\Lambda):\Ext^1_\Lambda(\mathcal{F}({}_\Lambda\Delta),M)=0\}$ (see \cite{AHLU}). In addition, they proved that there exists a tilting $\Lambda$-module $T=\{T(i)\}^{n}_{i=1}$, called the characteristic tilting module, such that $\add(T)=\F(\leftidx{_\Lambda}{\Delta})\cap\F(\leftidx{_\Lambda}{\Delta})^{\bot_1}$ (see also \cite{Platzeck Reiten}). Moreover, $(\leftidx{_\Lambda}{\Delta},\{T(i)\}^{n}_{i=1},\leq)$ is a stratifying system and $(\leftidx{_\Lambda}{\overline{\nabla}},\{T(i)\}^{n}_{i=1},\leq)$ is a proper costratifying system. This motivates the following problem:
given a set $\Q=\{Q(i)\}^{t}_{i=1}$ of pairwise non-isomorphic indecomposable $\Lambda$-modules and a linear order $\leq$ on $\{1,2,...t\}$, how does the existence of a proper costratifying system $(\Psi,\Q,\leq)$ relate to the existence of a stratifying system $(\Theta,\Q,\leq)$? In section 5 we study this question, proving two theorems. The first one gives, for a given proper costratifying system $(\Psi,\Q,\leq)$, necessary and sufficient conditions for the existence of a family $\Theta=\{\Theta(i)\}^{t}_{i=1}$ in $\modu(\Lambda)$ such that $(\Theta,\Q,\leq)$ is a stratifying system. The second one is the reciprocal result.

\section{Preliminaries}

Throughout this paper $\Lambda$ is an \emph{artin $R$-algebra}, where $R$ is a commutative artinian ring. The term \lq $\Lambda$-module\rq  \ means \emph{finitely generated left $\Lambda$-module}. The category of finitely
generated left $\Lambda$-modules is denoted by $\modu\,(\Lambda)$ and the full
subcategory of finitely generated projective $\Lambda$-modules by
$\proj\,(\Lambda)$. Let $\ind\;(\Lambda)$ denote the full subcategory of $\modu\;(\Lambda)$ whose objects consist of chosen representatives of isomorphism classes of indecomposable modules in $\modu\;(\Lambda)$. For $\Lambda$-modules $M$ and $N$,
$\Tr_M\,(N)$ is the \emph{trace} of $M$ in $N$, that is, $\Tr_M\,(N)$ is
the $\Lambda$-submodule of $N$ generated by the images of all
morphisms from $M$ to $N$.
Let $D:\modu\,(\Lambda)\rightarrow\modu\,(\Lambda^{op})$ denote the \emph{usual duality} for artin algebras, and $*$ denote the functor $\Hom_\Lambda(-,\Lambda):\modu\,(\Lambda)\rightarrow\modu\,(\Lambda^{op})$. Then $*$ induces a duality from $\proj\,(\Lambda)$ to
$\proj\,(\Lambda^{op}).$
For a given natural number $t,$ we set $[1,t]=\{1,2,\cdots, t\}.$
\

Let $\Lambda$ be an algebra and $n$ be the rank of the Grothendieck group $K_0\,(\Lambda)$. We fix a linear order $\leq$ on $[1,n]$ and a representative set ${}_\Lambda P=\{{}_{\Lambda}P(i)\;:\;i\in[1,n]\}$ containing one module of each iso-class of indecomposable projective $\Lambda$-modules. The injective envelope of the simple $\Lambda$-module ${}_{\Lambda}S(i)=\Top\,({}_{\Lambda}P(i))$ is denoted by ${}_{\Lambda}I(i)$. For the opposite algebra $\Lambda^{op}$, we always consider the representative set ${}_{\Lambda^{op}} P=\{{}_{\Lambda^{op}}P(i):i\in[1,n]\}$ of indecomposable projective $\Lambda^{op}$-modules, where ${}_{\Lambda^{op}}P(i)=({}_{\Lambda}P(i))^*$ for all $i\in[1,n]$. So, with these choices in mind, we recall the definition (see \cite{R,DR,ADL,Dlab}) of the following classes of $\Lambda$-modules:
\

The set of \emph{standard $\Lambda$-modules} is $\leftidx{_\Lambda}{\Delta}=\{
\leftidx{_\Lambda}{\Delta}(i):i\in[1,n]\},$ where
$\leftidx{_\Lambda}{\Delta}(i)={}_{\Lambda}P(i)/\Tr_{\oplus_{j>i}\,{}_{\Lambda}P(j)}\,({}_{\Lambda}P(i))$. Then,
${}_{\Lambda}{\Delta}(i)$ is the largest factor module of
${}_{\Lambda}P(i)$ with composition factors only amongst
${}_{\Lambda}S(j)$ for $j\leq i.$ The set of \emph{costandard
$\Lambda$-modules} is
${}_{\Lambda}\!\nabla=D({}_{\Lambda^{op}}{\Delta}),$
where the pair $({}_{\Lambda^{op}}P,\leq)$ is used to compute
${}_{\Lambda^{op}}\Delta.$
 \

The set of \emph{proper standard $\Lambda$-modules} is
${}_{\Lambda}\overline{\Delta}=\{{}_{\Lambda}\overline{\Delta}(i):i\in[1,n]\},$
where
${}_{\Lambda}\overline{\Delta}(i)={}_{\Lambda}P(i)/\Tr_{\oplus_{j\geq
i}\,{}_{\Lambda}P(j)}\,(\rad\,{}_{\Lambda}P(i))$. Then,
${}_{\Lambda}\overline{\Delta}(i)$ is the largest factor module of
${}_{\Lambda}{\Delta}(i)$ satisfying the multiplicity condition
$[{}_{\Lambda}\overline{\Delta}(i):S(i)]=1.$ The set of \emph{proper
costandard $\Lambda$-modules} is
${}_{\Lambda}\!\overline{\nabla}=D({}_{\Lambda^{op}}\overline{\Delta}),$
where the pair $({}_{\Lambda^{op}}P,\leq)$ is used to compute
${}_{\Lambda^{op}}\overline{\Delta}.$
 \

Let $\F({}_{\Lambda}{\Delta})$ be the subcategory of $\modu\,(\Lambda)$ consisting of the
$\Lambda$-modules having a ${}_{\Lambda}{\Delta}$-filtration, that is,
a filtration  $0=M_0\subseteq M_1\subseteq\cdots\subseteq M_s=M$ with factors $M_{i+1}/M_i$
isomorphic to a module in ${}_{\Lambda}{\Delta}$ for all $i$. The algebra $\Lambda$ is
a \emph{standardly stratified algebra} with respect to the linear
order $\leq$ on the set $[1,n]$, if
$\proj\,(\Lambda)\subseteq\F({}_{\Lambda}{\Delta})$ (see \cite{ADL,Dlab,CPS}).
\

Let $\Lambda$ be an algebra and $\C$ a class of objects in $\modu\,(\Lambda)$. For
each natural number $n$, we set ${}^{\perp_n}\C=\{M\in\modu\,(\Lambda): \Ext_\Lambda^n(M,-)|_{\C}=0\}$
and ${}^{\perp}\C=\cap_{n>0}\,{}^{\perp_n}\C.$ The notions of $\C^{\perp_n}$
and $\C^{\perp}$ are introduced similarly.
\

In addition, we recall that the \emph{$\Lambda$-length of the class} $\C$ is
$\ell_\Lambda(\C)=\sup\,\{\ell_\Lambda(C)\;:\;C\in
\C\}$, where $\ell_\Lambda(C)$ stands for the length of the $\Lambda$-module $C$.
\

Finally, for a $\Lambda$-module $M$, $\add\;(M)$ denotes the full subcategory of $\modu\;(\Lambda)$ consisting of the direct summands of direct sums of copies  of $M$.

\section{Proper pre-costratifying systems}

We begin by recalling the definition of proper costratifying systems introduced in \cite{MPV}.

\begin{defin}
\label{def de sist. estricto}
Let $\Lambda$ be an artin $R$-algebra. A proper costratifying system
$(\Psi,\Q,\leq)$ of size $t$ in $\modu(\Lambda)$ consists of two
families of $\Lambda$-modules $\Psi=\{\Psi(i)\}_{i=1}^t$ and $\Q=\{
Q(i)\}_{i=1}^t,$ with $Q(i)$ indecomposable for all $i$, and a
linear order $\leq$ on the set $[1,t],$ satisfying the following
conditions.
\begin{itemize}
 \item[(a)] $\End_\Lambda(\Psi(i))$ is a division ring for all $i\in[1,t]$.
 \item[(b)] $\Hom_\Lambda(\Psi(i),\Psi(j))=0$ if $i<j.$
 \item[(c)] For each $i\in[1,t],$ there is an exact sequence
$$\varepsilon_i: 0\longrightarrow Z(i)\longrightarrow Q(i)\overset{\beta_i}{\longrightarrow}\Psi(i)\longrightarrow0,$$
with $Z(i)\in \F(\{\Psi(j):j\leq i\}).$
 \item[(d)] $\Q\subseteq{}^{\perp_1}\Psi$, that is, $\Ext_\Lambda
 ^{1}(Q(i),-)|_{\Psi}=0$ for any $i\in[1,n].$
\end{itemize}
\end{defin}

%

Our aim is to show an existence result for proper costratifying systems along the lines of the existence result of K. Erdmann and C. S\'aenz for stratifying
systems, mentioned in the introduction of this paper. As we said there, we have a different situation for proper costratifying systems. In fact, the existence of a family
$\Psi=\{\Psi(i)\}_{i=1}^{t}$ such that $\Hom_\Lambda(\Psi(i),\Psi(j))=0=\Ext^{1}_\Lambda(\Psi(i),\Psi(j))$ for $i<j$, does not guarantee the
existence of a proper costratifying system $(\Psi,\Q,\leq),$ even if we assume that $\End_\Lambda(\Psi(i))$ is a division ring for all $i\in[1,t],$
as we show in the following example.

\begin{example}
\rm Let $\Lambda$ be the path algebra of the quiver

$$\xymatrix{
 {\circ} \ar@<3pt>[r] \ar@<-3pt>[r]
       & {\circ}
               }$$\vspace{-0.7cm}
$$\xymatrix{1&2}$$
\

\noindent with the natural order $1\leq 2$. Consider $\Psi=\{\Psi(1)=k\overset{1}{\underset{1}{\rightrightarrows}} k\}$. Then $\Psi(1)$
belongs to a tube $\mathcal{T}$ in the AR quiver of $\Lambda$. It follows from the structure of the tube that $\mathcal{T}\subseteq\F(\Psi).$
 So $\F(\Psi)$ contains modules of arbitrary length. Let now $Q(1)\in\F(\Psi)$. Then, there exists a monomorphism $\Psi(1)\rightarrow Q(1)$,
so that $Q(1)\in\mathcal{T}$ or $Q(1)$ is preinjective. If there is, moreover, a map $Q(1)\rightarrow\Psi(1)$, then $Q(1)\in\mathcal{T}$.
So, if $Q(1)$ satisfies (c) in \ref{def de sist. estricto}, then $Q(1)=k^{i}\overset{1}{\underset{J_{i,1}}{\rightrightarrows}} k^{i}$ for some $i\geq1$ (here $J_{i,1}$ denotes the Jordan block corresponding to the eigenvalue $1$), and therefore $\Ext^{1}_\Lambda(Q(1),\Psi(1))\neq0$
since there is a non-split exact sequence
$$0\longrightarrow\Psi(1)\longrightarrow (k^{i+1}\overset{1}{\underset{J_{i+1,1}}{\rightrightarrows}} k^{i+1})\longrightarrow Q(1)\longrightarrow 0.$$
This proves that there is no module $Q(1)\in\F(\Psi)$ satisfying (c) and (d) in Definition \ref{def de sist. estricto}. Thus, there is no proper
costratifying system $(\Psi,\{Q(1)\},\leq)$ for the family $\Psi$.
\end{example}

With the previous example in mind, we prove our desired existence result for proper costratifying systems, assuming as an additional hypothesis that the length $\ell_\Lambda(X)$ of the indecomposable modules $X$ in $\F(\Psi)$ is uniformly bounded. We also introduce the following notion.

\begin{defin}
\label{def de pre-sist. estricto}
Let $\Lambda$ be an artin $R$-algebra. A proper pre-costratifying system
$(\Psi,\leq)$ of size $t$ in $\modu\,(\Lambda)$ consists of a
family of $\Lambda$-modules $\Psi=\{\Psi(i)\}_{i=1}^t$ and a
linear order $\leq$ on the set $[1,t],$ satisfying the following
conditions.
\begin{itemize}
 \item[(a)] $\End_\Lambda(\Psi(i))$ is a division ring for all $i\in[1,t]$.
 \item[(b)] $\Hom_\Lambda(\Psi(i),\Psi(j))=0$ for $i<j.$
 \item[(c)] $\Ext^{1}_\Lambda(\Psi(i),\Psi(j))=0$ for $i<j$.
\end{itemize}
\end{defin}

\begin{remark} \label{compaginamientoFiltrados}
\rm It is not hard to prove, by using an inductive argument, that whenever (b) in the above definition holds for the family $\Psi$, then
$\Hom_\Lambda(X,Y)=0$, for $X\in\F(\{\Psi(j):j< s\})$, $Y\in\F(\{\Psi(j):j\geq s\})$ and $s\in[1,t]$.
\end{remark}

\begin{defin}
\label{la longitud maxima ordenada}
Let $(\Psi,\leq)$ be a proper pre-costratifying system of size $t$ in $\modu(\Lambda),$  and let $X\in \F(\Psi)$. An
\textbf{ordered $\Psi$-filtration of $X$ of length $n$} is a $\Psi$-filtration
$$\mathfrak{O}_{n,X}\;:\quad 0=X_0\subseteq X_1\subseteq X_2\subseteq\cdots\subseteq X_{n-1}\subseteq X_n=X,$$
with factors $X_k/X_{k-1}\simeq\Psi(i_k)$ for all $1\leq k\leq n$ and $i_1\leq i_2\leq\cdots\leq i_n.$ In this case, we also say that
$\text{\rm \textbf{min}}\,(\mathfrak{O}_{n,X})=i_1$.
\end{defin}

\begin{remark}
\rm  The hypothesis about $\Ext^{1},$ assumed in Definition \ref{def de pre-sist. estricto} (c), ensures the existence of ordered
$\Psi$-filtrations (see \cite[Lemma 3.10]{MPV}). Moreover, any $\Psi$-filtration of $X$ can be rearranged to an ordered one with the
same $\Psi$-composition factors.
\end{remark}

We aim to show that $\text{min}\,(\mathfrak{O}_{n,X})$ depends only on $X$, and not on the chosen ordered $\Psi$-filtration. In order to do that,
we start with the following lemma.

\begin{lemma}\label{LemmaDesigualdad} Let $(\Psi,\leq)$ be a proper pre-costratifying system of size $t$ in $\modu(\Lambda),$  and consider
the following diagram in $\modu\,(\Lambda)$
$$\begin{CD}
@. @. @. \Psi(\tilde{s})\\
@. @. @. @VVqV\\
\varepsilon\;:\;0 @>>> \Psi(s) @>>> X @>p>> Y @>>> 0\, ,
\end{CD}$$
\noindent where $\varepsilon$ is an exact sequence. If $X\in\F(\{\Psi(j)\;:\;j\geq s\})$ and $q\neq 0,$ then $\tilde{s}\geq s.$
\end{lemma}
\begin{proof} Let $X\in\F(\{\Psi(j)\;:\;j\geq s\})$ and $q\neq 0.$ Consider the following pull-back diagram
$$\begin{CD}
 \varepsilon'\;:\;0 @>>> \Psi(s) @>>> E @>p>> \Psi(\tilde{s}) @>>> 0 \\
 @. @| @VVV   @VVqV \\
\varepsilon\;:\;0 @>>> \Psi(s) @>>> X @>p>> Y @>>> 0\, .
\end{CD}$$
\noindent If $\varepsilon'$ does not split then $\Ext^1_{\Lambda}(\Psi(\tilde{s}),\Psi(s))\neq 0,$ and so  $\tilde{s}\geq s.$ Assume now
that $\varepsilon'$ splits. Therefore, there exists $q':\Psi(\tilde{s})\to X$ such that $pq'=q,$ and since $q\neq 0,$ it follows that
$0\neq q'\in\Hom_\Lambda(\Psi(\tilde{s}), X).$ Thus, from Remark \ref{compaginamientoFiltrados}, we conclude that $\tilde{s}\geq s$ since
$X\in\F(\{\Psi(j)\;:\;j\geq s\}).$
\end{proof}

\begin{proposition}\label{MinimosCoinciden} Let $(\Psi,\leq)$ be a proper pre-costratifying system of size $t$ in $\modu\,(\Lambda),$  and
let $X\in\F(\Psi).$ If $\mathfrak{O}_{m,X}$ and $\mathfrak{O'}_{n,X}$ are ordered $\Psi$-filtrations of $X,$ then
$\min\,(\mathfrak{O}_{m,X})=\min\,(\mathfrak{O'}_{n,X}).$
\end{proposition}
\begin{proof} Let $\min\,(\mathfrak{O}_{m,X})=i_1$ and $\min\,(\mathfrak{O}_{n,X})=j_1.$ So, we have the exact sequences
$\varepsilon_1\;:\;0\to\Psi(i_1)\stackrel{\nu_1}{\to}X\stackrel{p_1}{\to}Y\to 0$ and
$\varepsilon_2\;:\;0\to\Psi(j_1)\stackrel{\nu_2}{\to}X\stackrel{p_2}{\to}Z\to 0$ in $\modu(\Lambda).$ Let $q:=p_1\nu_2:\Psi(j_1)\to Y.$
If $q\neq 0,$ then from Lemma \ref{LemmaDesigualdad} it follows that $j_1\geq i_1$ since  $X\in\F(\{\Psi(j):j\geq i_1\}).$ On the other hand,
if $q=0,$ there is some $\nu:\Psi(j_1)\to\Psi(i_1)$ such that $\nu_1\nu=\nu_2\neq 0.$ Thus $0\neq\nu\in\Hom_\Lambda(\Psi(j_1),\Psi(i_1))$ and
 therefore $j_1\geq i_1.$ Similarly, it can be seen that $j_1\leq i_1;$ proving the result.
\end{proof}

The above proposition allows us to introduce the following definition.
\begin{defin}
\label{definicion minimo de X}
Let $(\Psi,\leq)$ be a proper pre-costratifying system of size $t$ in $\modu(\Lambda),$  and
let $X\in\F(\Psi)$. Then $\mini\,(X):=\mini\,(\mathfrak{O}_{n,X})$ for any ordered $\Psi$-filtration $\mathfrak{O}_{n,X}$ of $X$, of length $n.$
\end{defin}

It follows directly, from the definition given above, that the following statement holds.

\begin{remark} \label{obs. sobre el minimo de X}
\rm Let $(\Psi,\leq)$ be a proper pre-costratifying system of size $t$ in $\modu(\Lambda),$  and
let $X\in\F(\Psi)$. Then, there exists an exact sequence $0\rightarrow\Psi(\mini\,(X))\rightarrow X\rightarrow X'\rightarrow0$ in $\F(\Psi),$ satisfying the following two
conditions:
 \begin{itemize}
     \item [(a)] $\min\,(X')\geq\mini\,(X)$.
\vspace{.2cm}
     \item [(b)] If $X$ has an ordered $\Psi$-filtration of length $n$, then $X'$ has an ordered $\Psi$-filtration of length $n-1$.
     \end{itemize}
\end{remark}

\section{Constructing the family $\Q=\{Q(i)\}_{i=1}^t$}

 Let $(\Psi,\leq)$ be a proper pre-costratifying system of size
 $t$ in $\modu\,(\Lambda).$ The aim of this section is to prove the existence
of a family of $\Lambda$-modules $\Q=\{Q(i)\}_{i=1}^t$ such that $(\Psi,\Q,\leq)$ is a proper costratifying system. The proof will be done
under the hypothesis that $\ell_\Lambda(\ind(\F(\Psi)))<\infty$. The following lemmas will be used in the sequel.

\begin{lemma}
\label{lema suc. X, Y1+Y2,Z indesc} Let $\varepsilon:0\longrightarrow X\overset{\left( \alpha_1 \atop \alpha_2
\right) }{\longrightarrow} Y_1\oplus Y_2\overset{(\beta_1\ \beta_2 )}{\longrightarrow} Z\longrightarrow0$ be a non-split exact sequence in
$\modu\,(\Lambda)$, where $X$ and $Z$ are indecomposable $\Lambda$-modules and $Y_1\neq 0$, $Y_2\neq 0$. Then
$\beta_1\alpha_1=-\beta_2\alpha_2\neq 0$.
\end{lemma}
\begin{proof} Suppose that $\beta_1\alpha_1=-\beta_2\alpha_2= 0$. Then, it can be proven that $X=\Ker\,(\alpha_1)\oplus\Ker\,(\alpha_2).$ Hence,
using the fact that $X$ is indecomposable, it follows that either $\Ker\,(\alpha_1)=0$ or $\Ker\,(\alpha_2)=0.$ We may assume that
$\Ker\,(\alpha_1)=0$ (the proof for the other case is very similar). Therefore
$X=\Ker\,(\alpha_2)$ and so $\alpha_2=0$. Thus $Z\simeq\Coker\,(\left( \alpha_1 \atop 0 \right))=Y_1/\alpha_1\,(X)\oplus Y_2$. Furthermore, since
$Z$ is indecomposable and $Y_2\neq 0$, we have that $\alpha_1(X)= Y_1$. Hence $\alpha_1$ is an isomorphism and
$(\alpha^{-1}_{1}\ 0)\left( \alpha_1 \atop 0 \right)=1_X$. Therefore $\varepsilon$ splits, which is a contradiction.
\end{proof}

\begin{lemma}\label{lema auxiliar morfismos}
Let $\C$ be a class of objects in $\modu(\Lambda)$ closed under extensions, and let $\{\beta_i:X_i\rightarrow X_{i-1}\}_{i=1}^n$ be a family of epimorphisms in $\modu(\Lambda).$ If $\Ker\,(\beta_i)\in\C$ for all $i\in[1,n],$ then $\Ker(\beta_1\cdots\beta_n)\in\C$.
\end{lemma}
\begin{proof}
We proceed by induction on $n$. It is immediate that the lemma holds for $n=1$.\\
Let $n>1$. Now, we consider the following exact and commutative diagram
$$\begin{CD}
@. 0 @. 0 @. \\
@. @VVV  @VVV @. \\
@. \Ker(\beta_n) @= \Ker(\beta_n) @. @.\\
@. @VVV @VVV @. \\
0 @>>> \Ker(\beta_1\cdots\beta_n) @>>> X_n @>>> X_0 @>>> 0 \\
@. @VVV @VV\beta_nV @| \\
0 @>>> \Ker(\beta_1\cdots\beta_{n-1}) @>>> X_{n-1} @>\beta_1\cdots\beta_{n-1}>> X_0 @>>> 0 \\
@. @VVV  @VVV @. @.\\
@. 0 @. 0. @. @.\\
\end{CD}$$
\noindent By hypothesis $\Ker(\beta_n)\in\mathcal{C}$, and by induction $\Ker(\beta_1\cdots\beta_{n-1})\in\mathcal{C}$. Thus, since $\mathcal{C}$ is closed under extensions, we have that $\Ker(\beta_1\cdots\beta_n)\in\mathcal{C}$ and this finishes the proof.
\end{proof}
\

To carry on the construction of the family $\Q=\{Q(i)\}_{i=1}^t,$ let us start with a family $\Psi=\{\Psi(1)\}$ having just one element,
such that $\End_\Lambda(\Psi(1))$ is a division ring. We need to construct an indecomposable module $Q(1)$ and an exact sequence
$\varepsilon_1:0\rightarrow Z(1)\rightarrow Q(1)\rightarrow\Psi(1)\rightarrow 0$, with $Z(1)\in\F(\{\Psi(1)\}),$ and  such that
$\Ext^{1}_\Lambda(Q(1),\Psi(1))=0$. We will do this through successive extensions in the following way:
\medskip
\begin{itemize}
\item[$\bullet$] If $\Ext^{1}_\Lambda(\Psi(1),\Psi(1))=0,$ we choose $Q(1)=\Psi(1)$.
\medskip
\item[$\bullet$] If $\Ext^{1}_\Lambda(\Psi(1),\Psi(1))\neq 0,$ we consider a non-split exact sequence $ 0\rightarrow \Psi(1)\rightarrow X_1\overset{\beta_1}{\rightarrow} X_0\rightarrow 0,$ where $X_0=\Psi(1).$
In case $\Ext^{1}_\Lambda(X_1,\Psi(1))=0,$ we choose $Q(1)=X_1$. Otherwise, we iterate the above procedure and we find non-split exact
sequences $0\rightarrow \Psi(1)\rightarrow X_{i+1}\overset{\beta_{i+1}}{\rightarrow} X_i\rightarrow 0$ for $i\geq 1$.
Since all $X_i\in\F(\Psi)$, $\ell_\Lambda(X_{i+1})>\ell_\Lambda(X_i)$ and we are assuming that $\ell_\Lambda(\ind(\F(\Psi)))<\infty$, this process must
stop. Thus, there is a natural $n$ such that $\Ext^{1}_\Lambda(X_n,\Psi(1))=0,$ and we choose $Q(1)=X_n$. If $X_n$ is indecomposable,
the exact sequence $0\rightarrow Z(1)\rightarrow X_n\overset{\beta_1\cdots\beta_n}{\longrightarrow}\Psi(1)\rightarrow0,$ with
$Z(1)=\Ker(\beta_1\cdots\beta_n),$ satisfies the required conditions because
$Z(1)\in\F(\Psi)$ (see Lemma \ref{lema auxiliar morfismos}).
\end{itemize}
\indent \indent Let $\F'(\{\Psi(1)\})$ be the family of all the modules $X_i$ which occur
using the above procedure. Then, we have actually proved the following result.

\begin{lemma}
\label{lema caso t=1 exist. sist.estrat.propio}
Let $\Psi(1)\in\modu(\Lambda)$ be such that $\End_\Lambda(\Psi(1))$ is a division ring and $\ell_\Lambda(\ind(\F(\{\Psi(1)\})))<\infty$. Then, there
exists a non-split exact sequence $0\rightarrow Z(1)\rightarrow Q(1)\rightarrow\Psi(1)\rightarrow0$ such that $Z(1)\in\F(\{\Psi(1)\})$ and
$Q(1)\in\leftidx{^{\bot_{1}}}{\Psi(1)}\cap\F'(\{\Psi(1)\})$.
\end{lemma}

So, we only need to show that the module $X_n$ constructed above is indecomposable. To do this, we will prove that each $X_i$ in
$\F'(\{\Psi(1)\})$ is indecomposable. More generally, for a family $\Psi=\{\Psi(1),\cdots,\Psi(t)\}$, we next define the class
$\F'(\Psi)$ and later, in Proposition \ref{X en F'es Xindescomponible}, we will prove that the  modules in $\F'(\Psi)$ are indecomposable.

\begin{defin}
\label{defin. de la clase F`}
Let $(\Psi,\leq)$ be a proper pre-costratifying system  of size $t$ in $\modu\,(\Lambda).$ For
each natural $n,$ we inductively define the class
$\F'_n(\Psi)$ as follows:
\begin{enumerate}
\item[(a)] $\F'_1(\Psi)=\Psi$, and
\item[(b)] suppose $n>1$ and $\F'_{n-1}(\Psi)$ is already defined. Then $X\in \F'_n(\Psi)$ if and only if either $X\in\F'_{n-1}(\Psi)$ or $X$
admits an ordered $\Psi$-filtration of length $n$ and a non-split exact sequence
$$\varepsilon_n\;:\quad 0\rightarrow\Psi(\min\,(X))\rightarrow X\rightarrow X'\rightarrow0\text{ with }X'\in\F'_{n-1}(\Psi).$$
\end{enumerate}
We set $\F'(\Psi)=\cup_{i\geq 1}\,\F'_i(\Psi)$.
\end{defin}

\begin{remark}\label{propiedadesF'}
\rm
\begin{enumerate}
\item Observe that $\F'_1(\Psi)\subseteq\F'_2(\Psi)\subseteq\cdots\subseteq\F'_n(\Psi)\subseteq
\F(\Psi),$ for any $n$.

\item In the above definition $\varepsilon_n$ is a sequence in $\F(\{\Psi(j):j\geq \mini\,(X)\})$, and $X'$ satisfies $\mini\,(X')\geq\mini\,(X)$, as follows from Lemma \ref{LemmaDesigualdad}.
\end{enumerate}
\end{remark}

For a proper pre-costratifying system $(\Psi,\leq)$ in $\modu\,(\Lambda),$ we will prove that any module $X$ in $\F'(\Psi)$ is
indecomposable. The proof will be done by induction on $n$ such that $X$ admits an ordered $\Psi$-filtration of length $n$. To do this,
we will use that $\Coker\,(\eta)$ belongs to $\F'(\{\Psi(j):j\geq \text{\rm min}(X)\})\cup\{0\}$,
for any $Y\in\F'(\{\Psi(j):j\geq \text{\rm min}(X)\})$ and any arbitrary monomorphism $\eta:\Psi(\text{\rm min}(X))\rightarrow Y$. So, we start by studying further properties of monomorphisms with domain $\Psi(i)$ for some $i\in[1,t]$.

\begin{defin} For any $K\in\modu\,(\Lambda),$ we consider the class $\mathbb{M}_K$ of all $X\in\modu(\Lambda)$ satisfying the following
property: $$\textrm{For all}\,f\in\Hom_\Lambda(K,X),\quad f \textrm{ is either zero or a monomorphism}.$$
\end{defin}

\begin{proposition}\label{lema suc. exacta X,Y,Z y Z con propiedad M_i}
Let $0\rightarrow K\overset{\alpha}{\longrightarrow} X\overset{\beta}{\longrightarrow} Z\rightarrow0$ be an exact sequence
in $\modu\,(\Lambda),$ and let $L\in\modu(\Lambda).$ Then, the following statements hold.
\begin{enumerate}
 \item[(a)] If $K\in\mathbb{M}_L$ and $Z\in\mathbb{M}_L$, then $X\in\mathbb{M}_L.$
 \item[(b)] If $\Hom_\Lambda(L,K)=0$ and $\eta:L\rightarrow X$ is non-zero, then $\beta\eta\neq 0.$
 \item[(c)] If $Z\in\mathbb{M}_L$ and $\eta:L\rightarrow X$ is such that $\beta\eta\neq 0,$ then $\eta\,(L)\cap\alpha\,(K)=0.$
 \end{enumerate}
\end{proposition}
\begin{proof}
(a) Assume that $K\in\mathbb{M}_L$ and $Z\in\mathbb{M}_L,$ and let $0\neq f:L\rightarrow X$. We next prove that $f$ is a monomorphism. Indeed, if $\beta f=0,$ then there exists $f':L\to K$ such that $f=\alpha f'.$ Hence $f'\neq 0$
and so $f'$ is a monomorphism, getting in this case that $f$ is a monomorphism. On the other hand, in case $\beta f:L\to Z$  is non-zero, it follows that $\beta f$ is a monomorphism, and then $f$ is so.
\

(b) The proof is straightforward.
\

(c) Let $Z\in\mathbb{M}_L$ and $\eta:L\rightarrow X$ be such that $0\neq\beta\eta:L\to Z.$ In particular, since $Z\in\mathbb{M}_L,$ it
follows that $\beta\eta$ is a monomorphism.\\
We now prove that $\eta\,(L)\cap\alpha\,(K)=0$. Let $\lambda=\eta\,(x)=\alpha\,(y)$ with $x\in L$ and $y\in K$. Then
$\beta\eta\,(x)=\beta\alpha\,(y)=0$ and since $\beta\eta$ is a monomorphism, we conclude that $x=0$ and hence $\lambda=0$. Therefore
$\eta\,(L)\cap\alpha\,(K)=0$.
\end{proof}

\begin{proposition}\label{lema propiedad M-lambda}
 Let $(\Psi,\leq)$ be a proper pre-costratifying system of size $t$ in $\modu(\Lambda)$. If $X\in\F(\Psi)$ and $\lambda\leq\text{\rm min}(X)$, then $X\in\mathbb{M}_{\Psi(\lambda)}.$
\end{proposition}
\begin{proof}
Let $X\in\F(\Psi)$. We write $\text{\rm min}(X)=i_0$ for short and consider $\lambda\leq i_0$. We proceed by induction on $n$ such that $X$ admits an ordered $\Psi$-filtration of length $n$.\\
If $n=0,$ we have that $X=0$ and so $X\in\mathbb{M}_{\Psi(\lambda)}.$ On the other hand, if $n=1$ then $X\simeq\Psi(i_0) $ and $\lambda\leq i_0.$
Hence, the result follows directly from (a) and (b) of Definition \ref{def de pre-sist. estricto}.\\
Let $n>1$ and suppose that $X$ admits an ordered $\Psi$-filtration of length $n$. If $\lambda <i_0$
then by Remark \ref{compaginamientoFiltrados}, we get that $\Hom_\Lambda(\Psi(\lambda),X)=0$ and
so $X\in\mathbb{M}_{\Psi(\lambda)}.$ Assume now that $\lambda=i_0$. By Remark
\ref{obs. sobre el minimo de X}, we know that there exists an exact sequence
$$\varepsilon\;:\quad 0\rightarrow\Psi(i_0)\overset{\alpha}{\longrightarrow} X\overset{\beta}{\longrightarrow} X'\rightarrow 0$$
with $\text{\rm min}(X')\geq i_0$ and such that $X'$ has an ordered $\Psi$-filtration of length $n-1$. Then, by induction, $X'\in\mathbb{M}_{\Psi(\lambda)}$.  Now, by applying Proposition \ref{lema suc. exacta X,Y,Z y Z con propiedad M_i} (a) to $\varepsilon$, we obtain that $X\in\mathbb{M}_{\Psi(\lambda)}$.
\end{proof}

\begin{proposition}
\label{lema X en F',Xindescomponible}
 Let $(\Psi,\leq)$ be a proper pre-costratifying system of size $t$ in $\modu\,(\Lambda)$ and $n\in\mathbb{N}$. Then, for any $X\in\F'_n(\Psi)$ and any non-zero morphism $\eta:\Psi(\text{\rm min}(X))\rightarrow X$, we have that $X/\eta(\Psi(\text{\rm min}(X)))\in\F'_{n-1}(\Psi)\cup\{0\}$.
\end{proposition}
\begin{proof}
Let us consider $X\in\F'_n(\Psi)$, $i_0=\text{\rm min}(X)$ and a non-zero morphism $\eta:\Psi(i_0)\rightarrow X$. We proceed by induction on $n$.\\
If $n=1$ then $X=\Psi(i_0)$ and hence $\eta$ is an isomorphism, so
$X/\eta(\Psi(i_0))=0$.\\
Let $n>1$. If $X\in\F'_{n-1}(\Psi),$ then by induction we get that
$X/\eta(\Psi(\text{\rm min}(X)))\in\F'_{n-2}(\Psi)\cup\{0\}\subseteq\F'_{n-1}(\Psi)\cup\{0\}.$
Otherwise, there exists a non-split exact sequence
$$\varepsilon\;:\quad 0\rightarrow\Psi(i_0)\overset{\alpha}{\longrightarrow} X\overset{\beta}{\longrightarrow} X'\rightarrow 0,$$
with $X'\in\F'_{n-1}(\Psi).$ By Remark \ref{propiedadesF'} (2), we know that $\text{min}(X')\geq i_0$. Thus, by Proposition \ref{lema propiedad M-lambda}, we conclude that $X'\in\mathbb{M}_{\Psi(i_0)}$. If $\alpha(\Psi(i_0))=\eta(\Psi(i_0))$ then
$X/\eta(\Psi(i_0))\simeq X'\in\F'_{n-1}(\Psi)$.\\
Let $\alpha(\Psi(i_0))\neq\eta(\Psi(i_0))$. We assert that $0\neq\beta\eta:\Psi(i_0)\to X'.$ Indeed, if $\beta\eta=0,$ there exists
an $\eta':\Psi(i_0)\to\Psi(i_0)$ such that $\eta=\alpha\eta',$ hence $\eta'$ is an isomorphism since $\eta'\neq 0.$ Therefore
$\eta(\Psi(i_0))=\alpha(\eta'(\Psi(i_0)))=\alpha(\Psi(i_0)),$ contradicting our hypothesis and proving that $\beta\eta\neq 0.$ So, from $X'\in\mathbb{M}_{\Psi(i_0)}$ and Proposition \ref{lema suc. exacta X,Y,Z y Z con propiedad M_i} (c), we
conclude that $\beta\eta:\Psi(i_0)\to X'$ is a monomorphism and $\alpha(\Psi(i_0))\cap\eta(\Psi(i_0))=0$. Therefore, we get the following
exact and commutative diagram
$$\begin{CD}
@. @. 0 @. 0  @. \\
@. @. @VVV @VVV @. \\
{}\;\;\;0 @>>> 0 @>>> \Psi(i_0) @>>> \beta\eta(\Psi(i_0)) @>>> 0\\
@. @VVV @VV{\eta}V @VVV @.\\
\varepsilon\;:\; 0 @>>> \Psi(i_0)  @>{\alpha}>> X @>{\beta}>> X' @>>> 0 \\
@. @| @VVV @VVV @.\\
\varepsilon'\;:\; 0 @>>> \Psi(i_0)  @>>> X/\eta(\Psi(i_0)) @>>> X'/\beta\eta(\Psi(i_0)) @>>> 0 \\
@. @. @VVV @VVV @. \\
@. @. 0 @. 0 @.\, .
\end{CD}$$
Since $X'\in\F'_{n-1}(\Psi),$ we get by induction that $X'/\beta\eta(\Psi(i_0))\in\F'_{n-2}(\Psi)\cup\{0\}$. If $X'/\beta\eta(\Psi(i_0))=0$ then $X/\eta(\Psi(i_0))\simeq\Psi(i_0)\in\F'_1(\Psi)\subseteq\F'_{n-1}(\Psi)$. In case $X'/\beta\eta(\Psi(i_0))\in\F'_{n-2}(\Psi)$, since $X'/\beta\eta(\Psi(i_0))$ has an ordered $\Psi$-filtration of length $n-2$, we obtain that $X/\eta(\Psi(i_0))$ has an ordered $\Psi$-filtration of length $n-1$ (see $\varepsilon'$). Moreover, the sequence $\varepsilon'$ does not split (use the non-split sequence $\varepsilon$),
and therefore $X/\eta(\Psi(i_0))\in\F'_{n-1}(\Psi)$.
\end{proof}

\begin{proposition}
\label{X en F'es Xindescomponible}
 Let $(\Psi,\leq)$ be a proper pre-costratifying system of size $t$ in $\modu\,(\Lambda).$ Then, any  $\Lambda$-module in $\F'(\Psi)$ is indecomposable.
\end{proposition}
\begin{proof} Let $X\in\F'_n(\Psi)$ and $i_0=\text{min}(X)$. We show, by induction on $n$, that $X$ is an indecomposable $\Lambda$-module.
If $n=1$ then $X=\Psi(i_0)$ and hence $X$ is indecomposable.\\

Let $n>1$. If $X\in\F'_{n-1}(\Psi),$ then by induction we get that $X$ is indecomposable. Otherwise,
there is a non-split exact sequence
$\varepsilon:0\rightarrow\Psi(i_0)\overset{\alpha}{\longrightarrow} X\overset{\beta}{\longrightarrow}Z\rightarrow 0$ with $Z\in\F'_{n-1}(\Psi)$. Hence, it follows that $Z$ is indecomposable by induction.
\

Suppose that $X=X'\oplus X''$
with $X'\neq 0$ and $X''\neq 0.$ Hence $\alpha=\left( \alpha' \atop \alpha''\right):\Psi(i_0)\to X'\oplus X''$ and
$\beta=(\beta'\ \beta'' ):X'\oplus X''\to Z.$ Therefore, by Lemma \ref{lema suc. X, Y1+Y2,Z indesc}, $\beta'\alpha'(\Psi(i_0))=-\beta''\alpha''(\Psi(i_0))\neq 0$. Thus, by Proposition \ref{lema propiedad M-lambda}, we have that $\beta''\alpha'':\Psi(i_0)\rightarrow Z$ and $\beta'\alpha':\Psi(i_0)\rightarrow Z$ are monomorphisms. 
In particular, $\alpha'$ and $\alpha''$ are monomorphisms. Moreover, from $\beta'\alpha'(\Psi(i_0))=-\beta''\alpha''(\Psi(i_0))$, we get that the epimorphism $\beta:X\rightarrow Z$ induces an epimorphism
$$ X'/\alpha'(\Psi(i_0))\oplus X''/\alpha''(\Psi(i_0))\overset{\overline{\beta}}{\longrightarrow}Z/\beta'\alpha'(\Psi(i_0)).$$
We assert that $\overline{\beta}$ is an isomorphism. Indeed, it follows easily from the equalities
\[
\ell_\Lambda(X'/\alpha'(\Psi(i_0))\oplus X''/\alpha''(\Psi(i_0)))=\ell_\Lambda(X)-\ell_\Lambda(\Psi(i_0))-\ell_\Lambda(\Psi(i_0))=\]
\[
=\ell_\Lambda(Z)-\ell_\Lambda(\Psi(i_0))=\ell_\Lambda(Z/\beta'\alpha'(\Psi(i_0))).
\]

In what follows, we show that $Z/\beta'\alpha'(\Psi(i_0))$ is indecomposable. Suppose that  $Z/\beta'\alpha'(\Psi(i_0))=0.$ Then, by the isomorphism $\overline{\beta},$ we have $X'=\alpha'(\Psi(i_0)).$ Hence, the morphism  $((\alpha')^{-1}\ 0 ):X'\oplus X''\to\Psi(i_0)$ gives us that $\varepsilon$ splits, a contradiction, proving that $Z/\beta'\alpha'(\Psi(i_0))\neq 0.$ Now, we consider the exact sequence
$$\varepsilon':0\longrightarrow\Psi(i_0)\overset{ \beta'\alpha'
                                          }{\longrightarrow}Z\longrightarrow Z/\beta'\alpha'(\Psi(i_0))\longrightarrow0.$$
Since $Z\in\F'_{n-1}(\Psi)$, it follows from Proposition \ref{lema X en F',Xindescomponible} that $Z/\beta'\alpha'(\Psi(i_0))$ belongs to $\F'_{n-2}(\Psi)\cup\{0\}.$ So, by induction, we get that $Z/\beta'\alpha'(\Psi(i_0))$ is indecomposable.
\

Finally, using that the map $\overline{\beta}$ is an isomorphism, we have that $\alpha'(\Psi(i_0))=X'$ or $\alpha''(\Psi(i_0))=X'',$ and hence $\varepsilon$ splits, a contradiction. Therefore, $X$ is an indecomposable $\Lambda$-module.
\end{proof}

We are now in a position to prove the main result of this section.

\begin{theorem}
Let $(\Psi,\leq)$ be a proper pre-costratifying system of size $t$ in $\modu\,(\Lambda).$ If $\ell_\Lambda(\ind(\F(\Psi)))<\infty$ then there exists a
family $\Q=\{Q(i)\}^{t}_{i=1}$ of $\Lambda$-modules in $\F'(\Psi)$ such that $(\Psi,\Q,\leq)$ is a proper costratifying system of size $t$ in $\modu\,(\Lambda)$.
\end{theorem}

\begin{proof} Let $\ell_\Lambda(\ind(\F(\Psi)))<\infty.$ Since modules in $\F'(\Psi)$ are indecomposable (see Proposition \ref{X en F'es Xindescomponible}), the proof is completed by showing that there exists a family $\Q=\{Q(i)\}^{t}_{i=1}$ of $\Lambda$-modules in $\F'(\Psi),$ satisfying conditions (c) and (d) in Definition
\ref{def de sist. estricto}. To prove this, we proceed by induction on the size $t$ of $\Psi$. If $t=1,$ the result follows directly from
Lemma \ref{lema caso t=1 exist. sist.estrat.propio}.\\
Let $t>1$. For the sake of simplicity, we may assume that the order $\leq$ in $[1,t]$ is the natural one. So, we have that
$(\widetilde{\Psi},\leq),$ with $\widetilde{\Psi}=\{\Psi(2),...,\Psi(t)\},$ is a proper
pre-costratifying system of size $t-1$ in
$\modu\,(\Lambda).$ Therefore, from the inductive hypothesis applied to the smaller system,  we conclude the existence of a family
$\widetilde{Q}=\{\widetilde{Q}(i)\}^{t}_{i=2}$ in $\F'(\widetilde{\Psi})\subseteq\F'(\Psi)$ satisfying:
\begin{enumerate}
\item[($\widetilde{c}$)] for each $i\in[2,t],$ there is an exact sequence
 $$\widetilde{\varepsilon}_i:0\rightarrow \widetilde{Z}(i)\rightarrow \widetilde{Q}(i)\overset{\lambda_i}{\rightarrow}\Psi(i)\rightarrow0,$$
 with $\widetilde{Z}(i)\in\F(\{\Psi(j):2\leq j\leq i\});$ and
\item[($\widetilde{d}$)] $\widetilde{Q}\subseteq\leftidx{^{\bot_{1}}}{\widetilde{\Psi}}$.
\end{enumerate}
Now, consider the family with just one element $\{\Psi(1)\}.$ For this case, we have already proved the theorem. Thus, there exists an exact sequence $$ \varepsilon_1:0\rightarrow Z(1)\rightarrow Q(1)\overset{\beta}{\rightarrow} \Psi(1)\rightarrow 0,$$
\noindent with $Z(1)\in\F(\{\Psi(1)\}),$ $Q(1)\in\F'(\{\Psi(1)\})\subseteq\F'(\Psi)$ and $\Ext^{1}_\Lambda(Q(1),\Psi(1))=0$. Then $\varepsilon_1$ satisfies (c) in Definition \ref{def de sist. estricto}. Furthermore, since $Q(1)\in\F(\{\Psi(1)\})$ and $\Ext^{1}_\Lambda(\Psi(1),\Psi(j))=0$ for $j\geq 2,$ we have that $\Ext^{1}_\Lambda(Q(1),\Psi)=0$. Thus (d) in Definition \ref{def de sist. estricto} holds.
\

We next construct the required exact sequence $\varepsilon_i$ for each $i\in [2,t]$. If $\Ext^{1}_\Lambda(\widetilde{Q}(i),\Psi(1))=0$, then we set
$\widetilde{\varepsilon}_i=\varepsilon_i$ and $\widetilde{Q}(i)=Q(i).$\\
Suppose that $\Ext^{1}_\Lambda(\widetilde{Q}(i),\Psi(1))\neq 0$. Then there exists a non-split exact sequence $$\delta_1:0\rightarrow\Psi(1)\rightarrow X_1\overset{\beta_1}{\rightarrow}\widetilde{Q}(i)\rightarrow 0.$$
\noindent Since $\widetilde{Q}(i)\in\F'(\Psi)$, we have that $X_1\in\F'(\Psi)$. Moreover, $X_1\in\leftidx{^{\bot_{1}}}{\widetilde{\Psi}}$, as follows by applying $\Hom_\Lambda(-,\Psi(j))$ to $\delta_1$, with $j\in [2,t]$.
On the other hand, we have the exact sequence $\nu_1:0\rightarrow K_1\rightarrow X_1\overset{\lambda_{i}\beta_1}{\rightarrow} \Psi(i)\rightarrow 0$ where $K_1=\Ker(\lambda_{i}\beta_1)$. Then, by Lemma \ref{lema auxiliar morfismos}, $K_1\in\F(\{\Psi(j):j\leq i\})$. Thus, if $\Ext^{1}_\Lambda(X_1,\Psi(1))=0$, we conclude that $\varepsilon_i=\nu_1,$ with $Q(i)=X_1,$ satisfies the required conditions.\\
Assume now that $\Ext^{1}_\Lambda(X_1,\Psi(1))\neq 0$. Then there exists a non-split exact sequence $\delta_2:0\rightarrow\Psi(1)\rightarrow X_2\overset{\beta_2}{\rightarrow}X_1\rightarrow 0$. Thus $X_2\in\F'(\Psi)$ because $X_1\in\F'(\Psi)$. Analogously, from $X_1\in\leftidx{^{\bot_{1}}}{\widetilde{\Psi}}$ we have that $X_2\in\leftidx{^{\bot_{1}}}{\widetilde{\Psi}}$. Moreover, by Lemma \ref{lema auxiliar morfismos}, we have an exact sequence $\nu_2:0\rightarrow K_2\rightarrow X_2\overset{\lambda_{i}\beta_1\beta_2}{\rightarrow} \Psi(i)\rightarrow 0$ where $K_2\in\F(\{\Psi(j):j\leq i\}).$ Iterating this procedure, we obtain non-split exact sequences $\delta_j:0\rightarrow\Psi(1)\rightarrow X_j\overset{\beta_j}{\rightarrow} X_{j-1}\rightarrow 0$ with $X_j\in\F'(\Psi)\cap\leftidx{^{\bot_{1}}}{\widetilde{\Psi}}$, for $1\leq j\leq k$. Since $\ell_\Lambda(X_1)<\ell_\Lambda(X_2)<\cdots <\ell_\Lambda(X_k)$ and $\ell_\Lambda(\ind(\F(\Psi)))<\infty$, we eventually reach some $X_n\in\F'(\Psi)
 $ such that $\Ext^{1}_\Lambda(X_n,\Psi(1))=0,$ and then $X_n\in\leftidx{^{\bot_{1}}}{\Psi}.$ Therefore, the exact sequence $0\rightarrow Z(i)\rightarrow Q(i)\overset{\beta}{\rightarrow} \Psi(i)\rightarrow 0$ with $Q(i)=X_n $, $\beta=\lambda_i\beta_1\beta_2\cdots\beta_n$ and $Z(i)=\Ker(\beta)$ satisfies the required conditions.
\end{proof}

\begin{corollary}
Let $\Lambda$ be an artin algebra of finite representation type. Then any proper pre-costratifying system $(\Psi,\leq)$ of size $t$ in $\modu\,(\Lambda)$ defines a proper costratifying system $(\Psi,\Q,\leq)$ of size $t$ in $\modu\,(\Lambda)$.
\end{corollary}

\section{Some homological characterizations.}

\indent Throughout the rest of this paper, for $M$ in $\modu\,(\Lambda)$ we consider the artin algebra
$\Gamma_M=\End(\leftidx{_\Lambda}{M})^{op}$ and the functors
$$\modu\,(\Lambda)\underset{G_M}{\overset{F_M}{\rightleftarrows}}\modu\,(\Gamma),$$
where $F_M=\Hom_\Lambda(M,-)$ and $G_M=M\otimes_\Gamma-$, as well as the functors
$$\modu\,(\Lambda)\underset{\overline{G}_M}{\overset{\overline{F}_M}{\rightleftarrows}}\modu\,(\Gamma^{op}),$$
where $\overline{F}_M=\Hom_\Lambda(-,M)$ and $\overline{G}_M=\Hom_{\Gamma^{op}}(-,M)$. We
also have the functor $*=\Hom_{\Gamma}(-,\Gamma):\modu\,(\Gamma)\to \modu\,(\Gamma^{op}).$ This
functor induces a duality $*:\proj\,(\Gamma)\to \proj\,(\Gamma^{op}),$ whose quasi-inverse
$\Hom_{\Gamma^{op}}(-,\Gamma^{op})$ is also denoted by $*.$ Finally, we denote by $D$ the
usual duality for artin algebras. 
\

We recall that the functors $F_M$ and $G_M$ induce, by restriction, inverse equivalences between
$\add\,(M)$ and $\proj\,(\Gamma).$ Furthermore, the functors $\overline{F}_M$ and $\overline{G}_M$ induce, by restriction, inverse dualities between
$\add\,(M)$ and $\proj\,(\Gamma^{op}).$

\begin{lemma}
\label{lema gral.igualdad de proyectivos}
Let $M$ be in $\modu\,(\Lambda)$. Then the
following statements hold.
\begin{enumerate}
 \item[(a)] $(\overline{G}_M\circ*\circ F_M)|_{\add\,(M)}\simeq 1_{\add\,(M)}.$
 \vspace{.2cm}
 \item[(b)] $(*\circ F_M)|_{\add\,(M)}\simeq \overline{F}_M|_{\add\,(M)}.$
\end{enumerate}
\end{lemma}
\begin{proof}
(a) For any $X\in\add\,(M),$ we have the following natural isomorphisms $\overline{G}_M(F_M(X)^{*})=
\Hom_{\Gamma^{op}}(F_M(X)^{*},M)\simeq F_M(X)^{**}\otimes_{\Gamma^{op}}M\simeq F_M(X)\otimes_{\Gamma^{op}}M\simeq M\otimes_{\Gamma}F_M(X)\simeq G_M\circ F_M(X)\simeq X.$
\

(b) It follows from (a) since
$(\overline{F}_M\circ\overline{G}_M)|_{\proj\,(\Gamma^{op})}\simeq 1_{\proj\,(\Gamma^{op})}.$
\end{proof}

\

We next recall the definition of the class $C^{\wedge}_2(M),$ which was introduced by Platzeck and Pratti in \cite{Platzeck Pratti} (the notation, used in \cite{Platzeck Pratti}, for
 such a class is $C^{M}_2$). The objects in $C^{\wedge}_2(M)$ are the $\Lambda$-modules $X$  admitting an exact sequence in $\modu\,(\Lambda)$
$$M_2\rightarrow M_1\rightarrow M_0\rightarrow X\rightarrow 0$$
with $M_i\in\add\,(M)$, and such that the induced sequence
 $$F_M(M_2)\rightarrow F_M(M_ 1)\rightarrow F_M(M_0)\rightarrow F_M(X)\rightarrow 0$$
is exact in $\modu\,(\Gamma)$.
\

Dually, we define the class $C^{\vee}_2(M)$, consisting of
the $\Lambda$-modules $Z$
admitting an exact sequence in $\modu\,(\Lambda)$
$$0\rightarrow Z\rightarrow M_0\rightarrow M_1\rightarrow M_2,$$
with $M_i\in\add\,(M)$, and such that the induced sequence
$$\overline{F}_M(M_2)\rightarrow\overline{F}_M(M_1)\rightarrow\overline{F}_M(M_0)\rightarrow\overline{F}_M(Z)\rightarrow0$$
is exact in $\modu\,(\Gamma^{op})$.

\begin{remark}\label{relacionesFuntores} \rm Let $M$ be in $\modu\,(\Lambda).$ Then, since
$F_{DM}\circ D\simeq \overline{F}_M$ and $D\circ G_{DM}\simeq \overline{G}_M$, we have that $D(C^{\vee}_2(M))=C^{\wedge}_2(DM)$. Furthermore, $\Gamma_M\simeq \Gamma_{DM}^{op}$ as rings.
\end{remark}

\begin{proposition}
\label{Prop. equivalencia C2Check}
 Let $M$ be in $\modu\,(\Lambda)$, and let $\mathcal{C}$ be a class of objects in $\modu\,(\Lambda)$ such that  $M\in\mathcal{C}^{\bot_1}.$ If $\F(\mathcal{C})\subseteq C^{\vee}_2(M)$ then the restriction
 $\overline{F}_M|_{\F(\mathcal{C})}:\F(\mathcal{C})\rightarrow \F(\overline{F}(\mathcal{C}))$ is an exact duality with quasi inverse $\overline{G}_M|_{\F(\overline{F}(\mathcal{C}))}:\F(\overline{F}(\mathcal{C}))\rightarrow \F(\mathcal{C})$.
\end{proposition}

\begin{proof} It follows from Remark \ref{relacionesFuntores} that this result is dual to the statement of Theorem 2.10 in \cite{MPV}.
\end{proof}

For a given family $\underline{M}=\{M(i)\}^{t}_{i=1}$ of objects in $\modu(\Lambda),$ we set
$M=\bigoplus^{t}_{i=1}M(i)$. We now recall the definition of an Ext-injective stratifying system.

\begin{defin} \cite[Definition 1.1]{ES} An Ext-injective stratifying system $(\Theta,\underline{Y},\leq)$ of size $t$ in $\modu(\Lambda),$ consists
of two families of non-zero $\Lambda$-modules
$\Theta=\{\Theta(i)\}_{i=1}^t$ and $\underline{Y}=\{ Y(i)\}_{i=1}^t$,
with $Y(i)$ indecomposable for all $i$, and a linear order $\leq$ on the
set $[1,t]$, satisfying the following conditions.
\begin{itemize}
 \item[(a)] $\Hom_\Lambda(\Theta(i),\Theta(j))=0$ if $i>j$.
 \item[(b)] For each $i\in[1,t]$, there is an exact sequence
$$ 0\rightarrow\Theta(i)\rightarrow Y(i)\rightarrow Z(i)\rightarrow0,$$
with $Z(i)\in \F(\{\Theta(j):j<i\}).$
 \item[(c)] $\Ext_\Lambda^{1}(-,Y)|_{\Theta}=0.$
\end{itemize}
\end{defin}

\

Throughout the rest of the section $\underline{Y}$ denotes the family of indecomposable $\Lambda$-modules
$\underline{Y}=\{ Y(i)\}_{i=1}^t$, and we consider $\Gamma=\End(\leftidx{_\Lambda}{Y})^{op}$ and the functors $F=F_Y$, $G=G_Y$, $\overline{F}=\overline{F}_Y$ and $\overline{G}=\overline{G}_Y$.
We also consider the representative set ${}_{\Gamma^{op}}\overline{P}=\{{}_{\Gamma^{op}}\overline{P}
(i)\;:\;i\in[1,t]\}$ of indecomposable projective $\Gamma^{op}$-modules, where
${}_{\Gamma^{op}}\overline{P}(i)=\overline{F}(Y(i))$ for all $i$.

\begin{remark}\label{CoherenciaProy}
\rm Notice that we have two representative sets of indecomposable projective $\Gamma^{op}$-modules: ${}_{\Gamma^{op}}\overline{P}$ (defined above) and ${}_{\Gamma^{op}}P=\{{}_{\Gamma^{op}}P(i)=F(Y(i))^{*}\;:\;i\in[1,t]\}$. By Lemma \ref{lema gral.igualdad de proyectivos} (b), it follows that ${}_{\Gamma^{op}}\overline{P}(i)\simeq {}_{\Gamma^{op}}P(i)$ for all $i$. In particular, if $\leq$ is a linear order on $[1,t]$ and
$\leq^{op}$ the opposite order to $\leq$, there is no difference (up to isomorphism) between the family of standard (proper standard) $\Gamma^{op}$-modules computed by using either $({}_{\Gamma^{op}}\overline{P},\leq^{op})$ or $({}_{\Gamma^{op}}P,\leq^{op})$.
\end{remark}

We will need the following result, proven by E. Marcos, O. Mendoza and C. S\'aenz in \cite{MMS1},
which gives necessary and sufficient conditions for a family of indecomposable $\Lambda$-modules
$\underline{Y}$ to admit an Ext-injective stratifying system
$(\Theta,\underline{Y},\leq)$.

\begin{theorem} \cite[Theorem 2.3]{MMS1}
\label{teo. equival. para sist. estrat. extinyectivos,MMS}
Let $\underline{Y}=\{Y(i)\}^{t}_{i=1}$ be a set of pairwise non-isomorphic indecomposable
$\Lambda$-modules, and $\leq$ be a linear order on $[1,t].$ Let ${}_{\Gamma^{op}}\Delta$ be the family of standard $\Gamma^{op}$-modules corresponding to the pair
$({}_{\Gamma^{op}}\overline{P},\leq^{op})$, where $\leq^{op}$ is the opposite order of
$\leq$. Then the following statements, (I) and (II), are equivalent.
\begin{enumerate}
 \item[(I)] There exists a family $\Theta=\{\Theta(i)\}^{t}_{i=1}$ in $\modu(\Lambda)$ such that $(\Theta,\underline{Y},\leq)$ is an Ext-injective stratifying system.
 \item[(II)]
  \begin{enumerate}
   \item[(a)] $\Ext^{1}_{\Gamma^{op}}(\leftidx{_{\Gamma^{op}}}{\Delta},\leftidx{_{\Gamma^{op}}}{Y})=0$ and the pair $(\Gamma^{op},\leq^{op})$ is a standardly stratified algebra.
   \item[(b)] There is a full subcategory $\mathcal{A}$ of $\modu(\Lambda)$, closed under extensions, and such that $\underline{Y} \subseteq \mathcal{A}\cap\mathcal{A}^{\bot_1}$.
   \item[(c)] The restriction $\overline{F}|_{\mathcal{A}}:\mathcal{A}\rightarrow \F(\leftidx{_{\Gamma^{op}}}{\Delta})$ is an exact duality with quasi inverse $\overline{G}|_{\F(\leftidx{_{\Gamma^{op}}}{\Delta})}:\F(\leftidx{_{\Gamma^{op}}}{\Delta})\rightarrow \mathcal{A}$.
  \end{enumerate}
\end{enumerate}
Moreover, if one of these equivalent conditions hold, then $\mathcal{A}$ is uniquely determined (up to equivalences) by the family $\underline{Y}$. More precisely, $\mathcal{A}\simeq\F(\Theta)$ and $\Theta(i)\simeq\overline{G}(\leftidx{_{\Gamma^{op}}}{\Delta(i)})$ for all $i\in[1,t].$
\end{theorem}

\
If we only assume that condition (a) in (II) holds then condition (b) in the definition of an Ext-injective stratifying system holds, as we prove next.

\begin{lemma}
\label{lema gral. suc. exacta Phi(i),Y(i), Z(i)}
 With the hypothesis and notations of Theorem \ref{teo. equival. para sist. estrat. extinyectivos,MMS}, let $\Phi=\{\Phi(i)\}_{i=1}^t$ where $\Phi(i)= \overline{G}(\leftidx{_{\Gamma^{op}}}{\Delta(i)})$ for all $i.$  If the pair $(\Gamma^{op},\leq^{op})$ is a standardly stratified algebra and $\Ext^{1}_{\Gamma^{op}}(\leftidx{_{\Gamma^{op}}}{\Delta},\leftidx{_{\Gamma^{op}}}{Y})=0,$ then the following conditions hold.

\begin{enumerate}
   \item[(a)] For each $i\in[1,t]$, there is an exact sequence
$$ 0\rightarrow\Phi(i)\rightarrow Y(i)\rightarrow Z(i)\rightarrow0,$$
with $Z(i)\in \F(\{\Phi(j):j<i\}).$
   \item[(b)] If, moreover, $\Ext^{1}_\Lambda(\Phi,Y)=0$ then, for each $M\in\F(\Phi)$, there exists an exact sequence
$$0\rightarrow M\rightarrow Y'\rightarrow M'\rightarrow0$$
in $\F(\Phi)$ with $Y'\in\add\,(Y)$. In particular $\F(\Phi)\subseteq C^{\vee}_2(Y)$.
\end{enumerate}
\end{lemma}
\begin{proof}
(a) Suppose that $\Ext^{1}_{\Gamma^{op}}(\leftidx{_{\Gamma^{op}}}{\Delta},\leftidx{_{\Gamma^{op}}}{Y})=0$. Then, $\overline{G}|_{\F(\leftidx{_{\Gamma^{op}}}{\Delta})}:\F(\leftidx{_{\Gamma^{op}}}{\Delta})\rightarrow\modu(\Lambda)$ is exact on $\F(\leftidx{_{\Gamma^{op}}}{\Delta})$. Since $(\Gamma^{op},\leq^{op})$ is a standardly stratified algebra, we have for each $i\in[1,t]$ an exact sequence in $\F(\leftidx{_{\Gamma^{op}}}{\Delta})$
$$0\rightarrow U(i)\rightarrow\leftidx{_{\Gamma^{op}}}{\overline{P}(i)}\rightarrow \leftidx{_{\Gamma^{op}}}{\Delta(i)}\rightarrow0,$$
with $U(i)\in\F(\{\leftidx{_{\Gamma^{op}}}{\Delta(j)}:j>^{op}i\})$. By applying $\overline{G}$ to the above sequence, we get the following exact sequence in $\modu(\Lambda)$
$$0\rightarrow \overline{G}(\leftidx{_{\Gamma^{op}}}{\Delta(i)})\rightarrow\overline{G}(\leftidx{_{\Gamma^{op}}}{\overline{P}(i)})\rightarrow \overline{G}(U(i))\rightarrow0.$$
Since $\Phi(i)=\overline{G}(\leftidx{_{\Gamma^{op}}}{\Delta(i)})$, $\overline{G}(\leftidx{_{\Gamma^{op}}}{\overline{P}(i)})=\overline{G}(\overline{F}(Y(i)))\simeq Y(i)$ and also $\overline{G}(U(i))\in\F(\{\Phi(j):j<i\})$, the condition (a) follows.
\

(b) Let $M\in\F(\Phi)$. We consider a $\Phi$-filtration of $M$
$$0=M_0\subseteq M_1\subseteq ...\subseteq M_{n-1}\subseteq M_n=M,$$
where $M_k/M_{k-1}\simeq\Phi(i_k)$, for all $1\leq k\leq n$.\\
Using (a), Snake's Lemma and the fact that $\leftidx{_\Lambda}{Y}\in\F(\Phi)^{\bot_1}$, we get an exact and commutative diagram in $\F(\Phi)$
$$\begin{CD}@. 0 @. 0 @. 0\\
@.@VVV @VVV @VVV\\
0@>>> \Phi(i_1) @>>> M_2 @>>> \Phi(i_2) @>>> 0 \\
 @. @VVV  @VVV @VVV \\
0@>>>Y(i_1) @>>> Y(i_1)\oplus Y(i_2)  @>>> Y(i_2) @>>> 0 \\
@. @VVV @VVV @VVV\\
0@>>> Z(i_1) @>>> Z_2 @>>> Z(i_2) @>>> 0 \\
@.@VVV @VVV @VVV\\
@. 0 @. 0 @. 0
\end{CD}  $$
In particular, the middle vertical sequence $0\rightarrow M_2\rightarrow Y(i_1)\oplus Y(i_2)\rightarrow Z_2\rightarrow 0$ gives us the required exact sequence for $n=2$. The result follows by iterating this argument.
\end{proof}

Using Proposition \ref{Prop. equivalencia C2Check} and Lemma \ref{lema gral. suc. exacta Phi(i),Y(i), Z(i)} (b), we can prove the following result, where for a given family $\underline{Y}=\{ Y(i)\}_{i=1}^t$ of $\Lambda$-modules, we consider $\Gamma=\End(\leftidx{_\Lambda}{Y})^{op}$ and the functors $\overline{F}=\overline{F}_Y$ and $\overline{G}=\overline{G}_Y$.

\begin{theorem}
\label{prop. adicional. equival. para sist. estrat. extinyectivos,MMS}
Let $\underline{Y}=\{Y(i)\}^{t}_{i=1}$ be a set of pairwise non-isomorphic indecomposable
$\Lambda$-modules, and $\leq$ be a linear order on $[1,t].$ Let ${}_{\Gamma^{op}}\Delta$ be the family of standard $\Gamma^{op}$-modules, computed by using the pair
$({}_{\Gamma^{op}}\overline{P},\leq^{op}),$ where $\leq^{op}$ is the opposite order of
$\leq$ and ${}_{\Gamma^{op}}\overline{P}(i)=\overline{F}(Y(i))$ for all $i$. Then, there
exists a family $\Theta=\{\Theta(i)\}^{t}_{i=1}$ in $\modu(\Lambda)$ such that $(\Theta,
\underline{Y},\leq)$ is an Ext-injective stratifying system if and only if the following
conditions hold.
\begin{itemize}
   \item[(a)] The pair $(\Gamma^{op},\leq^{op})$ is a standardly stratified algebra.
   \item[(b)] $\Ext^{1}_{\Gamma^{op}}(\leftidx{_{\Gamma^{op}}}{\Delta},\leftidx{_{\Gamma^{op}}}{Y})=0=\Ext^{1}_\Lambda(\overline{G}(\leftidx{_{\Gamma^{op}}}{\Delta}),\leftidx{_\Lambda}{Y})$.
\end{itemize}
If these conditions hold, the Ext-injective stratifying system $(\Theta,
\underline{Y},\leq)$ is uniquely determined (up to isomorphism) and
$\Theta(i)\simeq \overline{G}(\leftidx{_{\Gamma^{op}}}{\Delta(i)})$ for all $i\in[1,t].$
\end{theorem}
\begin{proof}
Suppose that there exists a family $\Theta=\{\Theta(i)\}^{t}_{i=1}$ in $\modu(\Lambda)$ such that $(\Theta,\underline{Y},\leq)$ is an Ext-injective stratifying system. Then, by Theorem \ref{teo. equival. para sist. estrat. extinyectivos,MMS}, we get that there exists a full subcategory $\mathcal{A}$ of $\modu(\Lambda)$ such that $\overline{G}|_{\F(\leftidx{_{\Gamma^{op}}}{\Delta})}\subseteq\mathcal{A}$ and $\underline{Y} \subseteq \mathcal{A}^{\bot_1}$. Hence, $\Ext^{1}_\Lambda(\overline{G}(\leftidx{_{\Gamma^{op}}}{\Delta}),\leftidx{_\Lambda}{Y})=0$. The rest of the proof of (a) and (b) follows immediately from the same theorem.
\

Assume now that (a) and (b) hold. Since $(\Gamma^{op},\leq^{op})$ is a
standardly stratified algebra and $\Ext^{1}_{\Gamma^{op}}(\leftidx{_{\Gamma^{op}}}{\Delta},\leftidx{_{\Gamma^{op}}}{Y})=0$, the proof is completed by showing that (b) and (c) in Theorem \ref{teo. equival. para sist. estrat. extinyectivos,MMS} hold. Let $\mathcal{A}=\F(\overline{G}(\leftidx{_{\Gamma^{op}}}{\Delta}))$. By applying Lemma \ref{lema gral. suc. exacta Phi(i),Y(i), Z(i)} (a) and the equality $\Ext^{1}_\Lambda(\overline{G}(\leftidx{_{\Gamma^{op}}}{\Delta}),\leftidx{_\Lambda}{Y})=0$, we have that $\underline{Y} \subseteq \mathcal{A}\cap\mathcal{A}^{\bot_1}$. On the other hand, it follows from Lemma \ref{lema gral. suc. exacta Phi(i),Y(i), Z(i)} (b) that $\mathcal{A}\subseteq\C^{\vee}_2(Y)$. Then (c) in Theorem \ref{teo. equival. para sist. estrat. extinyectivos,MMS} holds by Proposition \ref{Prop. equivalencia C2Check}.
\end{proof}


An algebra $\Lambda$ is standardly stratified if and only if $\leftidx{_\Lambda}{\Lambda}$ is filtered by the standard modules. The fact that this is the case if and only if $D(\Lambda_\Lambda)$ is filtered by the corresponding proper costandard modules (see \cite{Dlab}, Proposition 2.2) and also the Remark \ref{CoherenciaProy} allow us to adapt the proof of Lemma \ref{lema gral. suc. exacta Phi(i),Y(i), Z(i)} to prove the following result, which will be useful later.

\begin{lemma}
\label{lema gral. (con el Tor) suc. exacta Z(i),Y(i), Phi(i)}
Let $\underline{Y}=\{Y(i)\}^{t}_{i=1}$ be a set of pairwise non-isomorphic indecomposable $\Lambda$-modules and $\leq$ be a linear order on $[1,t].$ Let $\leftidx{_{\Gamma^{op}}}{\overline{\nabla}}$ be the family of proper costandard $\Gamma^{op}$-modules computed by using the pair
$({}_{\Gamma^{op}}I,\leq^{op}),$ where ${}_{\Gamma^{op}}I$ is the representative set of injective
$\Gamma^{op}$-modules defined as ${}_{\Gamma^{op}}I(i)=D(F(Y(i)))$ for all $i.$ If $(\Gamma^{op},\leq^{op})$ is a standardly stratified algebra, $\Tor^{\Gamma}_{1}(Y,\leftidx{_\Gamma}{\overline{\Delta}})=0$ and $\Phi(i)= G(\leftidx{_\Gamma}{\overline{\Delta}}(i))$ for all $i$, then the following conditions hold.

\begin{enumerate}
   \item[(a)] For each $i\in[1,t]$, there is an exact sequence
$$ 0\rightarrow Z(i)\rightarrow Y(i)\rightarrow \Phi(i)\rightarrow0,$$
with $Z(i)\in \F(\{\Phi(j):j\leq i\}).$
   \item[(b)] If, moreover, $\Ext^{1}_\Lambda(Y,\Phi)=0$ then, for each $M\in\F(\Phi)$, there exists an exact sequence
$$0\rightarrow M'\rightarrow Y'\rightarrow M\rightarrow0$$
in $\F(\Phi)$ with $Y'\in\add\,(Y)$. In particular $\F(\Phi)\subseteq\C^{\wedge}_2(Y)$.
\end{enumerate}
\end{lemma}

\section{Relation between proper costratifying systems and Ext-injective stratifying systems.}

Generalizing results of C. M. Ringel for quasi-hereditary algebras, it is proven in \cite{AHLU} and \cite{Platzeck Reiten} that, for a standardly stratified algebra $(\Lambda,\leq)$, there exists a tilting $\Lambda$-module $T=\{T(i)\}^{n}_{i=1}$, called the characteristic tilting module, such that $\add(T)=\F(\leftidx{_\Lambda}{\Delta})\cap\F(\leftidx{_\Lambda}{\Delta})^{\bot_1}$ and the pair $(\End_\Lambda(T),\leq^{op})$ is again a standardly stratified algebra. Moreover, $(\leftidx{_\Lambda}{\Delta},\{T(i)\}^{n}_{i=1},\leq)$ is an Ext-injective stratifying system and $(\leftidx{_\Lambda}{\overline{\nabla}},\{T(i)\}^{n}_{i=1},\leq)$ is a proper costratifying system. This raises the question:\\
Given a set $\Q=\{Q(i)\}^{t}_{i=1}$ of pairwise non-isomorphic indecomposable $\Lambda$-modules and a linear order $\leq$ on $[1,t]$, how does the existence of a proper costratifying system $(\Psi,\Q,\leq)$ relate to the existence of an Ext-injective stratifying system $(\Theta,\Q,\leq)$?\\
This section is devoted to answer this question.\\
We use throughout the following notation. Let $\Q=\{Q(i)\}^{t}_{i=1}$ be a set of pairwise
non-isomorphic indecomposable $\Lambda$-modules, $\leq$ be a linear order on $[1,t],$ and
$Q=\bigoplus_{i=1}^{t}\,Q(i).$ We consider $\Gamma=\End(\leftidx{_\Lambda}{Q})^{op}$ and the functors $F=F_Q$, $G=G_Q$, $\overline{F}=\overline{F}_Q$ and $\overline{G}=\overline{G}_Q$.
According with Remark \ref{CoherenciaProy}, the family $\leftidx{_{\Gamma^{op}}}{\Delta}$
can be computed by using either $({}_{\Gamma^{op}}\overline{P},\leq^{op})$ or
$({}_{\Gamma^{op}}P,\leq^{op}),$ where ${}_{\Gamma^{op}}\overline{P}
(i)=\overline{F}(Q(i))$ and ${}_{\Gamma^{op}}P(i)=F(Q(i))^{*}$ for all $i$. We also consider the family  ${}_{\Gamma}{\overline{\Delta}}$, which is computed with the order $\leq^{op}$ on $[1,t]$, by using the representative set ${}_{\Gamma}P$ of projective indecomposable $\Gamma$-modules, where ${}_{\Gamma}P(i)=F(Q(i))$ for all $i$.

\subsection{From proper costratifying systems to Ext-injective stratifying systems.}

\

Let $(\Psi,\Q,\leq)$ be a proper costratifying system of size $t$ in $\modu\,(\Lambda)$. We recall from \cite[Theorem 4.3]{MPV}, that $(\Gamma^{op},\leq^{op})$ is a standardly stratified algebra and the restriction $F|_{\F(\Psi)}:\F(\Psi)\rightarrow \F(\leftidx{_\Gamma}{\overline{\Delta}})$ is an equivalence with quasi inverse $G|_{\F(\leftidx{_\Gamma}{\overline{\Delta}})}:\F(\leftidx{_\Gamma}{\overline{\Delta}})\rightarrow \F(\Psi)$.
\

In our next theorem we state, for a given proper costratifying system $(\Psi,\Q,\leq)$, necessary and sufficient conditions for the existence of a family $\Theta=\{\Theta(i)\}^{t}_{i=1},$ in $\modu(\Lambda)$ such that $(\Theta,\Q,\leq)$ is an Ext-injective stratifying system.

\begin{theorem}
\label{teo. dado sistema estrat.propio, existe sist. estart. extinyectivo.}
Let $(\Psi,\Q,\leq)$ be a proper costratifying system of size $t$ in $\modu\,(\Lambda)$. Then, the following conditions are equivalent.
\begin{enumerate}
\item[(a)] There exists a family $\Theta=\{\Theta(i)\}^{t}_{i=1}$ in $\modu(\Lambda)$ such that $(\Theta,\Q,\leq)$ is an Ext-injective stratifying system.
\item[(b)] $\Ext^{1}_{\Gamma^{op}}(\leftidx{_{\Gamma^{op}}}{\Delta},\leftidx{_{\Gamma^{op}}}{Q})=0=\Ext^{1}_{\Lambda}(\overline{G}(\leftidx{_{\Gamma^{op}}}{\Delta}),\leftidx{_\Lambda}{Q})$.
\end{enumerate}
If these conditions hold, the system $(\Theta,\Q,\leq)$ is uniquely determined (up to isomorphism) and $\Theta(i)\simeq\overline{G}(\leftidx{_{\Gamma^{op}}}{\Delta}(i))$ for all $i$.
\end{theorem}
\begin{proof}
 The proof follows directly from Theorem \ref{prop. adicional. equival. para sist. estrat. extinyectivos,MMS}, since by \cite[Theorem 4.3]{MPV} we know that $(\Gamma^{op},\leq^{op})$ is a standardly stratified algebra.
\end{proof}

\begin{example}
\rm In the following example we give a proper costratifying system $(\Psi,\Q,\leq)$ and apply Theorem \ref{teo. dado sistema estrat.propio, existe sist. estart. extinyectivo.} to get an Ext-injective stratifying system $(\Theta,\Q,\leq)$.\\
Let $\Lambda$ be the path algebra given by the quiver
$$\underset{1}{\circ}\longrightarrow\underset{2}{\circ}\longrightarrow\underset{3}{\circ},$$
with the natural order $1\leq 2\leq 3.$ Consider $\Psi=\{\Psi(1)=3,\;\Psi(2)=1,\;\Psi(3)={1 \atop 2}\}$ and $\Q=\{Q(1)=3,\;
Q(2)=1,\; Q(3)=\small\begin{array}{c}1 \\2 \\
3\end{array}\}.$ Then $(\Psi,\Q,\leq)$ is a proper costratifying system of size 3 in $\modu(\Lambda)$. In this case, the algebra
$\Gamma^{op}=\End_\Lambda(Q)$ is given by the quiver
$$\underset{1}{\circ}\overset{\varepsilon}{\longrightarrow}\underset{3}{\circ}\overset{\mu}{\longrightarrow}\underset{2}{\circ}$$
with the relation $\mu\varepsilon=0$. We consider $(\Gamma^{op},\leq^{op})$, where
$3\leq^{op}2\leq^{op}1.$ Then the corresponding standard modules are
$\leftidx{_{\Gamma^{op}}}{\Delta}=\{\leftidx{_{\Gamma^{op}}}{\Delta}(1)=\small\begin{array}{c}1 \\
 3\end{array}, \leftidx{_{\Gamma^{op}}}{\Delta}(2)=2, \leftidx{_{\Gamma^{op}}}{\Delta}(3)=3\}$, and $\leftidx{_{\Gamma^{op}}}{Q}=\small\begin{array}{c}
                                                           3 \\
                                                           2
                                                         \end{array}\oplus\; 3\;\oplus\small\begin{array}{c}
                                                                      1 \\
                                                                      3
                                                                    \end{array}$.
 Since $\leftidx{_{\Gamma^{op}}}{\Delta}(1)$ and $\leftidx{_{\Gamma^{op}}}{\Delta}(2)$ are projective modules, and $\leftidx{_{\Gamma^{op}}}{Q}(1)$ and $\leftidx{_{\Gamma^{op}}}{Q}(3)$ are injective modules, we get that $\Ext^{1}_{\Gamma^{op}}(\leftidx{_{\Gamma^{op}}}{\Delta},\leftidx{_{\Gamma^{op}}}{Q})=0$. It remains to check that $\Ext^{1}_{\Lambda}(\overline{G}(\leftidx{_{\Gamma^{op}}}{\Delta}),\leftidx{_\Lambda}{Q})=0$. Since $\overline{G}(\leftidx{_{\Gamma^{op}}}{\Delta}(1))=\overline{G}(\leftidx{_{\Gamma^{op}}}{Q}(3))=\leftidx{_\Lambda}{P}(3)$ and  $\overline{G}(\leftidx{_{\Gamma^{op}}}{\Delta}(3))=\overline{G}(\leftidx{_{\Gamma^{op}}}{Q}(2))=\leftidx{_\Lambda}{P}(2)$ are projective modules, then the required condition holds. On the other hand, a computation shows that $\overline{G}(\leftidx{_{\Gamma^{op}}}{\Delta}(2))=1$ and $\Ext^{1}_{\Lambda}(\overline{G}(\leftidx{_{\Gamma^{op}}}{\Delta}(2)),\leftidx{_\Lambda}{Q})=0$. Thus, by Theorem \ref{teo. dado sistema estrat.propio, existe sist. estart. extinyectivo.}, $(\Theta,\Q,\leq)$ is an Ext-injective stratifying system with $\Theta(1)\simeq\overline{G}(\leftidx{_{\Gamma^{op}}}{\Delta}(1))=3$, $\Theta(2)\simeq\overline{G}(\leftidx{_{\Gamma^{op}}}{\Delta}(2))=1$ and $\Theta(3)\simeq\overline{G}(\leftidx{_{\Gamma^{op}}}{\Delta}(3))=\small\begin{array}{c}
                                                                      2 \\
                                                                      3
                                                                    \end{array}$.

\end{example}

\subsection{From Ext-injective stratifying systems to proper costratifying systems.}

 \

 The aim of this subsection is to find, for a given Ext-injective stratifying system
 $(\Theta,\Q,\leq)$, necessary and sufficient conditions for the existence of a family
 $\Psi=\{\Psi(i)\}^{t}_{i=1}$ in $\modu(\Lambda)$ such that $(\Psi,\Q,\leq)$ is a proper
 costratifying system. We first consider the more general problem of studying the existence of
 such a proper costratifying system $(\Psi,\Q,\leq)$ assuming only that $\Q=\{Q(i)\}^{t}_{i=1}$
 is a family of pairwise non-isomorphic indecomposable $\Lambda$-modules. In analogy with  \cite[Theorem 2.3]{MMS1} and Theorem \ref{prop. adicional. equival. para sist. estrat. extinyectivos,MMS}, we prove the following result.

\begin{theorem}
\label{teo. equival. para sist. estrat. propios}
Let $\Q=\{Q(i)\}^{t}_{i=1}$ be a set of pairwise non-isomorphic indecomposable $\Lambda$-modules. Then, the following conditions (I), (II) and (III) are equivalent.
\begin{enumerate}
 \item[(I)]
  \begin{enumerate}
   \item[(a)] $(\Gamma^{op},\leq^{op})$ is a standardly stratified algebra and $\text{\rm Tor}^{\Gamma}_1(Q,\leftidx{_\Gamma}{\overline{\Delta}})=0$.
   \item[(b)] There is a full subcategory $\mathcal{B}$ of $\modu\,(\Lambda)$, closed under extensions, and such that $Q \subseteq \mathcal{B}\cap\leftidx{^{\bot_1}}{\mathcal{B}}$.
   \item[(c)] The restriction $F|_{\mathcal{B}}:\mathcal{B}\rightarrow \F(\leftidx{_\Gamma}{\overline{\Delta}})$ is an exact equivalence of categories with $G|_{\F(\leftidx{_\Gamma}{\overline{\Delta}})}:\F(\leftidx{_\Gamma}{\overline{\Delta}})\rightarrow \mathcal{B}$ as a quasi inverse of $F|_{\mathcal{B}}$.
  \end{enumerate}
 \item[(II)] There exists a family $\Psi=\{\Psi(i)\}^{t}_{i=1}$ in $\modu(\Lambda)$ such that $(\Psi,\Q,\leq)$ is a proper costratifying system.
 \item[(III)]
 \begin{enumerate}
   \item[(a)] $(\Gamma^{op},\leq^{op})$ is a standardly stratified algebra.
   \item[(b)] $\text{\rm Tor}^{\Gamma}_1(Q,\leftidx{_\Gamma}{\overline{\Delta}})=0=\Ext^{1}_\Lambda(\leftidx{_\Lambda}{Q},G(\leftidx{_\Gamma}{\overline{\Delta}}))$.
\end{enumerate}
\end{enumerate}
Moreover, if these conditions hold, then $\mathcal{B}$ is uniquely determined by the family $Q$. More precisely, $\mathcal{B}\simeq\F(\Psi)$ and $\Psi(i)\simeq G(\leftidx{_\Gamma}{\overline{\Delta}}(i))$ for all $i\in[1,t]$.
\end{theorem}
\begin{proof}
\indent (I)$\Rightarrow$(II). Assume that (I) holds. Let $\Psi(i):= G(\leftidx{_\Gamma}{\overline{\Delta}}(i))$ for all $i\in[1,t]$. To prove that $(\Psi,\Q,\leq)$ is a proper costratifying system, we have to show that the conditions of \cite[Definition 3.1]{MPV} are satisfied.
\begin{itemize}
 \item[(a)] Since $\End_\Lambda(\Psi(i))\simeq\End_\Gamma(\leftidx{_\Gamma}{\overline{\Delta}}(i)),$ it follows that $\End_\Lambda(\Psi(i))$ is a division ring for all $i\in[1,t]$.
 \item[(b)] Since $\Hom_\Lambda(\Psi(i),\Psi(j))\simeq\Hom_\Gamma(\leftidx{_\Gamma}{\overline{\Delta}}(i),\leftidx{_\Gamma}{\overline{\Delta}}(j)),$ we get that\\ $\Hom_\Lambda(\Psi(i),\Psi(j))=0$ for $j<^{op}i.$
 \item[(c)] We have to prove that, for each $i\in[1,t],$ there is an exact sequence
$$ 0\rightarrow Z(i)\rightarrow Q(i)\rightarrow \Psi(i)\rightarrow0,$$
with $Z(i)\in \F(\{\Psi(j):j\leq i\})$. In fact, this follows by Lemma \ref{lema gral. (con el Tor) suc. exacta Z(i),Y(i), Phi(i)} (a) applied to $\Phi:=\Psi$ and $\underline{Y}:=\Q.$
 \item[(d)] $\Ext^{1}_\Lambda(Q,\Psi)=0$ since $\Psi=G(\leftidx{_\Gamma}{\overline{\Delta}})\subseteq\mathcal{B}$ and $Q \subseteq \mathcal{B}\cap\leftidx{^{\bot_1}}{\mathcal{B}}$.
\end{itemize}

\indent (II)$\Rightarrow$(III) Let $(\Psi,\Q,\leq)$ be a proper costratifying system. Then, by
\cite[Theorem 4.3]{MPV}, we get that $(\Gamma^{op},\leq^{op})$ is a standardly stratified algebra and  $\Psi(i)\simeq G(\leftidx{_\Gamma}{\overline{\Delta}}(i)).$  Thus, $\Ext^{1}_\Lambda(\leftidx{_\Lambda}{Q},G(\leftidx{_\Gamma}{\overline{\Delta}}))=0$, by definition of proper costratifying system. Finally, it follows from \cite[Proposition 2.5]{MPV} that\\ $F(\F(\Psi))\subseteq F(C^{\wedge}_2(Q))\subseteq\text{Ker Tor}^{\Gamma}_1(Q,-)$.\\
\indent (III)$\Rightarrow$(I) Let $\mathcal{B}:=\F(G(\leftidx{_\Gamma}{\overline{\Delta}}))$. The arguments in the proof of ``($\Leftarrow$)'' in Theorem \ref{prop. adicional. equival. para sist. estrat. extinyectivos,MMS}  can be adapted to this case, by using Lemma \ref{lema gral. (con el Tor) suc. exacta Z(i),Y(i), Phi(i)} and \cite[Theorem 2.10]{MPV}.
\end{proof}
\

We are now in a position to show the main result of this section, which is an analogue of Theorem \ref{teo. dado sistema estrat.propio, existe sist. estart. extinyectivo.} and we state in the next theorem.

\begin{theorem}
\label{teo. dado sistema estrat.extinyectivo, existe sist. estart. propio.}
Let $(\Theta,\Q,\leq)$ be an Ext-injective stratifying system of size $t$ in $\modu\,(\Lambda)$. Then, the following conditions are equivalent.
\begin{itemize}
\item[(a)] There exists a family $\Psi=\{\Psi(i)\}^{t}_{i=1}$ in $\modu\,(\Lambda)$ such that $(\Psi,\Q,\leq)$ is a proper costratifying system.
\item[(b)] $\Tor^{\Gamma}_{1}(Q,\leftidx{_\Gamma}{\overline{\Delta}})=0=\Ext^{1}_{\Lambda}(Q,G(\leftidx{_\Gamma}{\overline{\Delta}}))$.
\end{itemize}
If these conditions hold, the system $(\Psi,\Q,\leq)$ is uniquely determined (up to isomorphism) and $\Psi(i)\simeq G(\leftidx{_\Gamma}{\overline{\Delta}}(i))$ for all $i\in [1,t]$.
\end{theorem}
\begin{proof}
The proof follows immediately from Theorem \ref{teo. equival. para sist. estrat. propios} and
Remark \ref{CoherenciaProy}, since by \cite{ES} we get that $(\Gamma^{op},\leq^{op})$ is a standardly stratified algebra.
\end{proof}

\
The following results show that, for a given family $\Q$ of
$\Lambda$-modules over an algebra $\Lambda$ not necessarily standardly stratified, the simultaneous existence of $\Theta$ and $\Psi$ such that $(\Theta,\Q,\leq)$ is an Ext-injective stratifying system and $(\Psi,\Q,\leq)$ is a proper costratifying system, implies that $\Gamma^{op}=\End(\leftidx{_\Lambda}{Q})$ is a standardly stratified algebra and $\leftidx{_{\Gamma^{op}}}{Q}$ coincides with the characteristic tilting module associated to $\Gamma^{op}$.

\begin{corollary}
\label{corollary dado sist propio, existe Ext-iny. con condicion sobre tilting}
Let $(\Psi,\Q,\leq)$ be a proper costratifying system of size $t$ in $\modu\,(\Lambda)$. Then, the following conditions are equivalent, where ${}_{\Gamma^{op}}T$ is the characteristic tilting module associated to the standardly stratified algebra $(\Gamma^{op},\leq^{op})$.
\begin{enumerate}
\item[(a)] There exists a family $\Theta=\{\Theta(i)\}^{t}_{i=1}$ in $\modu(\Lambda)$ such that $(\Theta,\Q,\leq)$ is an Ext-injective stratifying system.
\item[(b)] $\leftidx{_{\Gamma^{op}}}{Q}\simeq {}_{\Gamma^{op}}T$ and $\Ext^{1}_{\Lambda}(\overline{G}(\leftidx{_{\Gamma^{op}}}{\Delta}),\leftidx{_\Lambda}{Q})=0$.
\end{enumerate}
\end{corollary}
\begin{proof}
\indent (a) $\Rightarrow$ (b) Suppose that $(\Theta,\Q,\leq)$ is an Ext-injective stratifying system. Since $(\Psi,\Q,\leq)$ is a proper costratifying system, it follows from Theorem \ref{teo. dado sistema estrat.propio, existe sist. estart. extinyectivo.} that
$\Ext^{1}_{\Lambda}(\overline{G}({}_{\Gamma^{op}}\Delta),{}_\Lambda Q)=0$. To prove that $\leftidx{_{\Gamma^{op}}}{Q}\simeq {}_{\Gamma^{op}}T$, it is enough to show that ${}_{\Gamma^{op}}Q\in \F({}_{\Gamma^{op}}\Delta)\cap\F({}_{\Gamma^{op}}\Delta)^{\bot_1}=\add({}_{\Gamma^{op}}T)$ (see \cite[Proposition 2.2]{AHLU}). From Theorem \ref{teo. dado sistema estrat.extinyectivo, existe sist. estart. propio.} (b) we have that $\Tor^{\Gamma}_{1}(Q,\leftidx{_\Gamma}{\overline{\Delta}})=0$. Then, by \cite[page 120]{CE}, $\Ext^{1}_{\Gamma^{op}}(\leftidx{_{\Gamma^{op}}}{Q},\leftidx{_{\Gamma^{op}}}{\overline{\nabla}})\simeq D(\Tor^{\Gamma}_{1}(Q,\leftidx{_\Gamma}{\overline{\Delta}}))=0$, that is, $Q\in{}^{\bot_1}\F(\leftidx{_{\Gamma^{op}}}{\overline{\nabla}})=\F(\leftidx{_{\Gamma^{op}}}{\Delta})$ (see \cite[Theorem 1.6]{AHLU}). On the other hand, we get from Theorem \ref{teo. dado sistema estrat.propio, existe sist. estart. extinyectivo.} (b) that $\Ext^{1}_{\Gamma^{op}}(\leftidx{_{\Gamma^{op}}}{\Delta},\leftidx{_{\Gamma^{op}}}{Q})=0$, that is, $Q\in\F(\leftidx{_{\Gamma^{op}}}{\Delta})^{\bot_1}$. Thus (b) holds.\\
\indent (b) $\Rightarrow$ (a) Since ${}_{\Gamma^{op}}T$ is the characteristic tilting module, then we get\\ $\Ext^{1}_{\Gamma^{op}}(\leftidx{_{\Gamma^{op}}}{\Delta},\leftidx{_{\Gamma^{op}}}{Q})\simeq\Ext^{1}_{\Gamma^{op}}(\leftidx{_{\Gamma^{op}}}{\Delta},\leftidx{_{\Gamma^{op}}}{T})=0$. Therefore, (a) follows from Theorem \ref{teo. dado sistema estrat.propio, existe sist. estart. extinyectivo.}.
\end{proof}

\
We illustrate this result with the following example.
\begin{example}
\rm Let $\Lambda$ be given by the quiver
$$
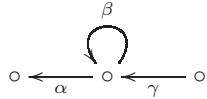
\xymatrix{
 {\circ}& \ar[l]^\alpha
        {\circ}  \ar@(ur,ul)[]_{\beta}
              & \ar[l]^\gamma{\circ} }$$\vspace{-0.7cm}
$$\xymatrix{1&2&3}$$
 with the relations $\beta^{2}=0$, $\beta\alpha=0$ and $\gamma\beta=0$.
Consider the natural order $1\leq2\leq 3$, and the sets
$\Psi=\{\Psi(1)=2,\;\Psi(2)=\small\begin{array}{c}3 \\2
\\1\end{array},\;\Psi(3)={2 \atop 1}\}$ and
$\Q=\{Q(1)=\small\begin{array}{c}2\\2\end{array},\; Q(2)=\small\begin{array}{c}3
\\2 \\1\end{array},\; Q(3)=\small\begin{array}{ccc} & 2 &  \\1 &  & 2\end{array}\}.$ Then
$(\Psi,\Q,\leq)$ is a proper costratifying system of size $3$ in
$\modu(\Lambda)$. In this case,
 $\Gamma^{op}=\End({}_\Lambda Q)$ is given by the quiver
$$\underset{1}{\circ}\overset{\mu}{\underset{\delta}{\rightleftarrows}}\underset{3}{\circ}\overset{\varepsilon}{\longrightarrow}
\underset{2}{\circ}$$
with the relations $\varepsilon\mu=0$ and $\mu\delta\mu=0$. By \cite[Theorem 4.3]{MPV}, we know
that $(\Gamma^{op},\leq^{op})$ is a standardly stratified algebra. The characteristic tilting module is \;
 $\small\leftidx{_{\Gamma^{op}}}{T}=\small\begin{array}{c} 3 \\1 \\3 \\1\end{array}\oplus\small\begin{array}{c}3 \\2\end{array}\oplus\; 3$, which
is not isomorphic to \;$\small\leftidx{_{\Gamma^{op}}}{Q}=\small\begin{array}{c}2 \\ 1 \end{array}\oplus\small\begin{tabular}{ccc}
                                                                                                  & 3 &   \\
                                                                                                1 &   & 2 \\
                                                                                                3 &   &   \\
                                                                                                1 &   &
                                                                                              \end{tabular}\oplus 2$.
Hence it follows from Corollary \ref{corollary dado sist propio, existe Ext-iny. con condicion sobre tilting} (b) that there exists no family $\Theta=\{\Theta(i)\}^{t}_{i=1}$ in $\modu(\Lambda)$ such that $(\Theta,\Q,\leq)$ is an Ext-injective stratifying system.
\end{example}

\

The following result is proven by using similar arguments to those used in the proof of Corollary \ref{corollary dado sist propio, existe Ext-iny. con condicion sobre tilting}.
\begin{corollary}
Let $(\Theta,\Q,\leq)$ be an Ext-injective stratifying system of size $t$ in $\modu\,(\Lambda)$. Then, the following conditions are equivalent, where ${}_{\Gamma^{op}}T$ is the characteristic tilting module associated to the standardly stratified algebra $(\Gamma^{op},\leq^{op})$.
\begin{enumerate}
\item[(a)] There exists a family $\Psi=\{\Psi(i)\}^{t}_{i=1}$ in $\modu\,(\Lambda)$ such that $(\Psi,\Q,\leq)$ is a proper costratifying system.
\item[(b)] $\leftidx{_{\Gamma^{op}}}{Q}\simeq {}_{\Gamma^{op}}T$ and $\Ext^{1}_{\Lambda}(Q,G(\leftidx{_\Gamma}{\overline{\Delta}}))=0$.
\end{enumerate}
\end{corollary}

 \

{\bf Acknowledgements.} The first author was partially
supported by the Project PAPIIT-Universidad Nacional Aut\'onoma de
M\'exico IN100810-3, and by the joint research project Mexico
(CONACyT)-Argentina(MINCYT): ``Homology, stratifying systems and
representations of algebras". The second author is a researcher from CONICET, Argentina. The second and third authors acknowledge partial support from Universidad Nacional del Sur and from CONICET.

\small
\markright{}

\footnotesize

\vskip3mm \noindent Octavio Mendoza Hern\'andez\\
Instituto de Matem\'aticas\\ Universidad Nacional Aut\'onoma de M\'exico.\\
Circuito Exterior, Ciudad Universitaria\\
C.P. 04510, M\'exico, D.F. MEXICO.\\ {\tt omendoza@matem.unam.mx}

\vskip3mm \noindent Mar\'{\i}a In\'es Platzeck\\
Instituto de Matem\'atica Bah\'{\i}a Blanca,\\
Universidad Nacional del Sur.\\
8000 Bahia Blanca, ARGENTINA.\\
{\tt platzeck@uns.edu.ar}

\vskip3mm \noindent     Melina Vanina Verdecchia\\
Instituto de Matem\'atica Bah\'{\i}a Blanca,\\
Universidad Nacional del Sur.\\
8000 Bahia Blanca, ARGENTINA.\\
{\tt mverdec@uns.edu.ar}

\end{document}